[1994/12/01]
\documentclass[draft]{article}
\title{Extremal Hypergraphs for Ryser's Conjecture:\\
Connectedness of Line Graphs of Bipartite Graphs}
\author{Penny Haxell\thanks{Department of Combinatorics and Optimization, University of Waterloo, Waterloo, ON, Canada. Partially supported by NSERC and by a Friedrich Wilhelm Bessel
 Award of the Alexander von Humboldt Foundation.} \quad Lothar Narins\thanks{Freie Universit\"at Berlin, Berlin, Germany. Supported by the Research Training Group 
 \emph{Methods for Discrete Structures} and the Berlin Mathematical School.} \quad 
Tibor Szab\'{o}\thanks{Freie Universit\"at Berlin, Berlin, Germany. Research partially supported by DFG within the Research Training Group \emph{Methods for Discrete Structures.}}}
\date{\today}

\usepackage{amsmath, amsthm, amssymb, tikz}
\usetikzlibrary{arrows, positioning, shapes}

\newtheorem{thm}{Theorem}[section]
\newtheorem{cor}[thm]{Corollary}
\newtheorem{lem}[thm]{Lemma}
\newtheorem{prop}[thm]{Proposition}

\newtheorem{conj}{Conjecture}
\newtheorem*{claim}{Claim}

\theoremstyle{definition}
\newtheorem{defn}[thm]{Definition}
\newtheorem*{assm}{Assumptions}

\theoremstyle{remark}

\numberwithin{equation}{section}

\DeclareMathOperator{\conn}{conn}
\DeclareMathOperator{\lk}{lk}

\newcommand{\N}{\mathbb{N}}
\newcommand{\R}{\mathbb{R}}

\newcommand{\cB}{\mathcal{B}}
\newcommand{\cC}{\mathcal{C}}
\newcommand{\cD}{\mathcal{D}}
\newcommand{\cE}{\mathcal{E}}

\newcommand{\cG}{\mathcal{G}}
\newcommand{\cH}{\mathcal{H}}
\newcommand{\cI}{\mathcal{I}}
\newcommand{\cK}{\mathcal{K}}

\newcommand{\cM}{\mathcal{M}}
\newcommand{\cN}{\mathcal{M}}
\newcommand{\cP}{\mathcal{P}}

\newcommand{\cS}{\mathcal{S}}
\newcommand{\cT}{\mathcal{T}}

\newcommand{\lset}[1]{\left\{#1\right\}}
\newcommand{\set}[2]{\left\{#1 : #2 \right\}}
\newcommand{\abs}[1]{\left|#1\right|}
\newcommand{\poly}[1]{\left\|#1\right\|}
\newcommand{\explode}{\divideontimes}
\newcommand{\link}[2]{\lk_{#1}(#2)}

\begin{document}
\maketitle
\markboth{Extremal Hypergraphs for Ryser's Conjecture}
{Connectedness of Line Graphs of Bipartite Graphs}
\renewcommand{\sectionmark}[1]{}

\begin{abstract}
In this paper we consider a natural extremal graph theoretic problem of topological
sort,  concerning the minimization of the (topological) connectedness of the independence 
complex of graphs in terms of its dimension. 
We observe that the lower bound $\frac{dim({\cal I}(G))}{2}-2$ on the connectedness 
of the independence complex ${\cal I}(G)$ of line graphs of bipartite graphs $G$
is tight. In our main theorem we characterize the extremal examples. 
Our proof of this characterization is based on topological machinery. 

Our motivation for studying this problem comes from a classical conjecture of Ryser.
Ryser's Conjecture states that any $r$-partite $r$-uniform hypergraph has a vertex cover of size at most $(r - 1)$-times the size of the largest matching. For $r = 2$, the conjecture is simply K\"onig's Theorem. It has also been proven for $r = 3$ by Aharoni using a beautiful topological argument. 
In a separate paper we characterize the extremal examples for the $3$-uniform case of 
Ryser's Conjecture (i.e., Aharoni's Theorem), and in particular resolve an old 
conjecture of Lov\'asz for the case of Ryser-extremal $3$-graphs. 

Our main result in this paper will provide us with valuable structural information for 
that characterization.
Its proof  is based on the observation that link graphs of 
Ryser-extremal $3$-uniform hypergraphs are 
{\em exactly} the bipartite graphs we study here.

%
\end{abstract}

\section{Introduction} \label{sec:introduction}

A hypergraph $\cH$ is a pair $(V, E)$, where $V = V(\cH)$ is the set of \emph{vertices}, and $E = E(\cH)$ is a {\bf multiset} of subsets of vertices called the \emph{edges} of $\cH$. The number of times a subset $e\subseteq V$ appears in $E$ is called the {\em multiplicity} of $e$. If the cardinality of every edge is $r$, we call $\cH$ an \emph{$r$-graph}. A $2$-graph is called a \emph{graph}. In our paper we mostly have no restriction on the multiplicity of edges; whenever we want to assume that each multiplicity is at most $1$, we will explicitly say \emph{simple hypergraph}, \emph{simple $r$-graph}, or \emph{simple graph}. An edge $e \in E$ is called \emph{parallel} to an edge $f \in E$ if their underlying vertex subsets are the same. In particular, every edge is parallel to itself.

Let $\cH$ be a hypergraph. A \emph{matching} in $\cH$ is a set of disjoint edges of $\cH$, and the \emph{matching number}, $\nu(\cH)$, is the size of the largest matching in $\cH$. If $\nu(\cH) = 1$, then $\cH$ is called \emph{intersecting}. A \emph{vertex cover} of $\cH$ is a set of vertices which intersects every edge of $\cH$. The size of the smallest vertex cover is called the \emph{vertex cover number} of $\cH$ and is denoted by $\tau(\cH)$. It is immediate to see that if $\cH$ is $r$-uniform, then the following bounds always hold:
\[
	\nu(\cH) \leq \tau(\cH) \leq r\nu(\cH).
\]

Both inequalities are easily seen to be tight for general hypergraphs. Ryser's Conjecture~\cite{ryser}, which appeared first in the late 1960's, states that the upper bound can be lowered by considering only $r$-partite hypergraphs. (An even stronger conjecture was made around the same time by Lov\'asz~\cite{lovasz}.) An $r$-graph is called \emph{$r$-partite} if its vertices can be partitioned into $r$ parts called \emph{vertex classes} such that every edge intersects each vertex class in exactly one vertex.

\begin{conj}[Ryser's Conjecture]
If $\cH$ is an $r$-partite $r$-graph, then
\[
	\tau(\cH) \leq (r - 1)\nu(\cH).
\]
\end{conj}

This conjecture turned out to be extremely difficult to attack. It is solved completely only for $r = 2$ and $3$, and a few partial results exists for values of $r \leq 9$. In particular, when $r = 2$, the conjecture is just the well known K\"onig's Theorem. 
For $r = 3$ Ryser's Conjecture was solved by Aharoni via topological methods~\cite{aharoni}.
The conjecture is wide open for $r \geq 4$. Haxell and Scott~\cite{haxellscott} have proven that for $r = 4, 5$ there is an $\epsilon > 0$ such that $\tau(\cH) \leq (r - \epsilon)\nu(\cH)$ for any $r$-partite $r$-graph $\cH$. 
The conjecture has been proven for intersecting hypergraphs when $r \leq 5$ by Tuza (\cite{tuza1},~\cite{tuza2}), with $r \geq 6$ still open. Recently Franceti\'c, Herke, McKay, and Wanless~\cite{fhmw-2015}  proved 
Ryser's Conjecture  for linear intersecting hypergraphs (i.e., hypergraphs where any two hyperedges intersect in exactly one
vertex) when $r\leq 9$.  
Fractional versions of the conjecture have also been studied, and it was shown by F\"uredi~\cite{fueredi} that $\tau^* \leq (r - 1)\nu$, and shown by Lov\'asz~\cite{lovasz} that $\tau \leq \frac{r}{2}\nu^*$, where $\tau^*$ and $\nu^*$ are the fractional vertex cover and matching numbers, respectively. Aharoni and Berger~\cite{aharoniberger} also formulated a generalization of the conjecture to matroids, which has been partially solved in a special case by Berger and Ziv~\cite{bergerziv}. Mansour, Song, and Yuster~\cite{mansoursongyuster} have found bounds on the minimum number of edges for an intersecting $r$-partite $r$-graph to be tight for Ryser's conjecture, and determined the exact numbers for the cases $r \leq 5$. Subsequently Aharoni, Bar\'at and Wanless~\cite{aharonibaratwanless} found these values for $r=6,7$ (see also~\cite{abuprok}).

One plausible approach to Ryser's Conjecture for $4$-graphs is via studying the $3$-uniform link hypergraphs. Given three of the four vertex classes $V_1$, $V_2$, $V_3$ of a $4$-partite $4$-graph $\cH$, the link hypergraph of $V_4$ in $\cH$ is the multiset of those $3$-element sets which are the intersection of an edge of $\cH$ with $V_1 \cup V_2 \cup V_3$. Having structural information on the links would be helpful in understanding the situation for $4$-graphs. Aharoni's proof however does not provide information on the $3$-graphs which are extremal for his theorem. Our eventual aim is to give a complete characterization of them. In a follow-up paper~\cite{HNS2} we prove the following theorem:

\begin{thm}[\cite{HNS2}] \label{thm:characterization}
Let $\cH$ be a $3$-partite $3$-graph. Then $\tau(\cH) = 2\nu(\cH)$ if and only if $\cH$ is a home-base hypergraph.
\end{thm}

Home-base hypergraphs are $3$-graphs with a certain restricted structure. For every $k$, however, there are infinitely many of them with matching number $k$ and covering number $2k$. Home-base hypergraphs are \emph{not} the focus of our current paper, hence their precise but somewhat technical definition will only be given in \cite{HNS2}.

We say that a $3$-partite $3$-graph $\cH$ is \emph{Ryser-extremal} if $\tau(\cH) = 2\nu(\cH)$. In the present paper we develop the necessary knowledge about the link graphs of Ryser-extremal $3$-graphs. First we show that these link graphs are extremal with respect to a natural extremal graph theoretic problem of topological nature. In our main theorem, Theorem~\ref{thm:bipartitegraphchar} below, we  characterize all those bipartite graphs that are extremal for this problem. The structure we derive from this characterization theorem will be an integral part of our proof of Theorem~\ref{thm:characterization} in \cite{HNS2}, nevertheless we find the extremal graph theory problem interesting in its own right.

\subsection{Connectedness of the Line Graphs of Bipartite Graphs}

The connectedness of the independence complex will be our main parameter to describe the line graphs of the link graphs of Ryser-extremal $3$-graphs.

Let $k \geq -1$ be an integer. A topological space $X$ is said to be \emph{$k$-connected} if for any integer $j$ with $-1 \leq j \leq k$, any continuous map from the $j$-dimensional sphere $S^j$ into the space $X$ can be extended to a continuous map from the $(j + 1)$-dimensional ball $B^{j + 1}$ to $X$. The \emph{connectedness} of $X$, denoted $\conn(X)$ is the largest $k$ for which $X$ is $k$-connected. 

A simplicial complex $\cK$ is a family of simplices in $\R^N$ such that (1) if $\tau$ is a face of a simplex $\sigma \in \cK$ then $\tau \in \cK$ and (2) if $\sigma, \sigma' \in \cK$ then $\sigma \cap \sigma'$ is a face of both $\sigma$ and $\sigma'$. The connectedness of a simplicial complex ${\cal K}$ is just the connectedness of its body $\poly{\cK}$ (the union of its simplices).

An \emph{abstract simplicial complex} $\cC$ is a simple hypergraph that is closed under taking subsets. The simple hypergraph consisting of the vertex sets of simplices of a simplicial complex $\cK$ (called the vertex scheme of $\cK$) is an abstract simplicial complex. Every abstract simplicial complex $\cC$ has a geometric realization, that is a simplicial complex whose vertex scheme is $\cC$. The geometric realization is unique up to homeomorphism. The connectedness of an abstract simplicial complex is just the connectedness of its geometric realization.

For a graph $G$, we define the \emph{independence complex} $\cI(G)$ to be the abstract simplicial complex on the vertices of $G$ whose simplices are the independent sets of $G$.  We will simply write $\conn(G)$ for $\conn(\cI(G))$, and refer to this as the \emph{connectedness} of $G$. We remark that in several other papers in this area (e.g. 
\cite{aharoniberger}) the parameter $\eta({\cal I}(G))=\conn(G)+2$ is used, 
which typically allows results to be stated more neatly. However, since in particular here 
we explicitly describe a significant amount of background material in terms of simplicial 
complexes, we choose to stick with the more traditional terminology.

One of the basic parameters of a simplicial complex is its \emph{dimension}, that is, the largest dimension that occurs among its simplices. The connectedness of an arbitrary simplicial complex, or even of an arbitrary graph's independence complex can be arbitrarily small while its dimension is large: just consider the complete bipartite graph $K_{d + 1, d + 1}$, having an independence complex with dimension $d$ and connectedness $-1$.

Comparing dimension and connectedness becomes more interesting if we introduce restrictions on the graphs we consider. For \emph{line graphs} for example, a lower bound on the connectedness in terms of the dimension is implicit in the work of Aharoni and Haxell~\cite{aharonihaxell} (see also~\cite{aharoniberger}). The \emph{line graph} $L(\cH)$ of a hypergraph $\cH$ is the simple graph $L(\cH)$ on the vertex set $E(\cH)$ with  $e, f \in V(L(\cH))$ adjacent if $e \cap f \neq \emptyset$. With foresight, we state the lower bound of~\cite{aharonihaxell} in a more general format, which will be necessary for our investigations.  Note that the dimension of the independence complex of a line graph of a hypergraph is just its matching number minus $1$.

\begin{thm} \label{thm:matchconn} 
Let $\cG$ be an $r$-graph, and let $J \subseteq L(\cG)$ be a subgraph of the line graph of $\cG$. Let $M \subseteq V(J)$ be a matching in $\cG$. Then
\[
	\conn(J) \geq \frac{\abs{M}}{r} - 2.
\]
In particular, for any graph $G$ we have $\conn(L(G)) \geq \frac{\nu(G)}{2} - 2$.
\end{thm}

Aharoni and Haxell~\cite{aharonihaxell} essentially proved that the connectedness of the line graph is at least the so called \emph{independent set domination number} $i\gamma$ of the line graph minus $2$ (where $i\gamma (G)$ is the smallest number $x$, such that every independent set of $G$ can be dominated with $x$ vertices.) Theorem~\ref{thm:matchconn} then follows from $i\gamma (L({\cal H})) \geq \frac{\nu({\cal H})}{r}$, which is immediate from the definitions. For completeness we give a proof of Theorem~\ref{thm:matchconn} in
the Appendix.

We begin our study of Ryser-extremal $3$-graphs with their link graphs.

\begin{defn} \label{def:linkg}
Let $\cH$ be a $3$-partite $3$-graph with parts $V_1$, $V_2$, and $V_3$. Let $S \subseteq V_i$ for some $i = 1, 2, 3$. Then the \emph{link graph} $\link{\cH}{S}$ is the bipartite graph with vertex classes $V_j$ and $V_k$ (where $\lset{i, j, k} = \lset{1, 2, 3}$) whose edge multiset is $\set{e \setminus V_i}{e \in E(\cH), e \cap V_i \subseteq S}$.
\end{defn}

Note that a pair of vertices appears as an edge in $\link{\cH}{S}$ with the same multiplicity as the number of edges in $\cH$ that contain it together with a vertex from $S$.

As implied by the proof of~\cite{aharoni},  the link graphs of Ryser-extremal $3$-graphs attest that Theorem~\ref{thm:matchconn} is optimal for $r = 2$, that is, among bipartite graphs they minimize the connectedness of the independence complex of the line graph.

\begin{thm} \label{thm:connoflink} 
If $\cH$ is a $3$-partite $3$-graph with vertex classes $V_1$, $V_2$, and $V_3$, such that $\tau(\cH) = 2\nu(\cH)$, then for each $i$ we have
\begin{enumerate}
	\renewcommand{\theenumi}{(\roman{enumi})}
	\renewcommand{\labelenumi}{\theenumi}
	\item \label{connoflink:conn} $\conn(L(\link{\cH}{V_i})) = \nu(\cH) - 2$.
	\item \label{connoflink:nu} $\nu(\link{\cH}{V_i}) = \tau(\cH)$.
\end{enumerate}
In particular 
\begin{equation}\label{eq:linkconn}
	\conn(L(\link{\cH}{V_i})) = \frac{\nu(\link{\cH}{V_i})}{2} - 2. 
\end{equation}
\end{thm}
We prove Theorem~\ref{thm:connoflink} explicitly in Section~\ref{sec:ryserforrequals3}. On the way, we give a reformulated proof of Aharoni's Theorem~\cite{aharoni}. We also mention here that  in~\cite{HNS2} we derive, as a consequence of Theorem~\ref{thm:characterization}, a sort of converse of Theorem~\ref{thm:connoflink}: \emph{every} bipartite graph which is optimal for Theorem~\ref{thm:matchconn} is the link of some Ryser-extremal $3$-graph.

In the main theorem of this paper, proven in Section~\ref{thelinkgraph}, we characterize those bipartite graphs which are extremal for Theorem~\ref{thm:matchconn} and hence we also obtain valuable structural information about the link graphs of Ryser-extremal $3$-graphs.

\begin{thm}\label{thm:bipartitegraphchar}
Let $G$ be a bipartite graph. Then $\conn(L(G)) = \frac{\nu(G)}{2} - 2$ if and only if $G$ has a collection of $\nu(G)/2$ pairwise vertex-disjoint subgraphs, each of them a $C_4$ or a $P_4$, such that every edge of $G$ is parallel to an edge of one of the $C_4$'s or is incident to an interior vertex of one of the $P_4$'s.
\end{thm}

To be precise, in this paper, we will in fact only prove the ``only if'' direction of this theorem. The other direction will be proven in~\cite{HNS2}, as it is not necessary for Theorem~\ref{thm:characterization}, and its proof makes use of the concept of home-base hypergraphs.

\subsection{Topological Tools}

The proofs of Theorems~\ref{thm:connoflink} and \ref{thm:bipartitegraphchar}, as well as the proof of Theorem~\ref{thm:characterization} (given in~\cite{HNS2}) use two tools to bound the topological connectedness of graphs.

The first one is a theorem of Meshulam~\cite{meshulam}, reformulated in a way that is particularly well-suited for our inductive arguments. Let $G$ be a graph, and let $e$ be an edge of $G$. We denote by $G - e$ the graph $G$ with the edge $e$ deleted. We denote by $G \explode e$ the graph $G$ with both endpoints of $e$ and their neighbors deleted. $G \explode e$ is called $G$ with $e$ \emph{exploded}. We will often write edges with endpoints $x$ and $y$ as $xy$.

\begin{thm} \label{thm:meshulam}
Let $G$ be a graph and let $e \in E(G)$. Then we have
\begin{equation} \label{eq:meshulamtheorem}
	\conn(G) \geq \min \lset{\conn(G - e), \conn(G \explode e) + 1}.
\end{equation}
\end{thm}

Meshulam proved a homological version of this theorem, where everywhere in the statement $\conn$ is replaced by the homological connectedness $\conn_H$. For completeness, in the Appendix we indicate how to extend Meshulam's argument using the approach of Adamaszek and Barmak~\cite{adamaszekbarmak} and obtain \eqref{eq:meshulamtheorem}. It is also possible to give a homology-free proof of Theorem~\ref{thm:meshulam} via triangulations along the lines of \cite{szabotardos} (cf~\cite{lotharsthesis}). Theorem~\ref{thm:meshulam} in this formulation but with a modified (non-topological) definition of $\conn$ was also stated in~\cite{haxell} and proved without direct reference to topology.

Our second tool is a special case of a theorem of Aharoni and Berger~\cite{aharoniberger}, which 
makes a direct connection between the size of the largest hypergraph matching and the connectedness of the link.

\begin{thm}[\cite{aharoniberger}] \label{thm:deficiencyrainbowmatching}
Let $d\geq 0$ be an integer and let $\cH$ be a $3$-uniform $3$-graph with vertex classes $V_1$, $V_2$, and $V_3$. If we have that $\conn(L(\link{\cH}{S})) \geq \abs{S} - d - 2$ for every $S \subseteq V_i$, then $\nu(\cH) \geq \abs{V_i} - d$.
\end{thm}

Since the statement of this theorem in \cite{aharoniberger} appears in a different 
terminology, we reproduce an argument using Sperner's Lemma in the Appendix.

\subsection{Extremal Problems with Many Extremal Structures}

In a typical extremal combinatorial problem, the greater the number of extremal configurations,  the less likely that a purely combinatorial argument will lead to a solution, since a proof eventually must consider all extremal structures, at least implicitly.  For our characterization problem, the number of extremal structures is infinite for every fixed value of the benchmark parameter. This is one of the few interesting cases in which the full characterization of the extremal structures of an extremal combinatorial problem with infinitely many extrema is known.

Sometimes, the difficulties posed by multiple extremal examples can be mitigated by realizing that the combinatorial problem, or rather its extremal structures, hide the features and concepts of another mathematical discipline in the background. 
In such cases, the simplest, or most efficient descriptions of extremal structures are not necessarily combinatorial, but might have to be formulated in another language, which could be algebraic, probabilistic, or, as in the present paper, topological. 

A simple example of a combinatorial problem of this sort is the famous Oddtown Theorem of 
Berlekamp~\cite{berlekamp}, where the right language is the one of linear algebra.
Even though the number of extremal set-systems is 
superexponential~\cite[Exercise 1.1.14]{babai-frankl} and the feasibility of their combinatorial characterization is questionable at best, 
they have a very simple 
linear algebraic description as the  orthogonal bases in $\mathbb{F}_2^n$.

Another prominent example is the extremal problem known as Sidorenko's 
Conjecture~\cite{simonovits, sidorenko}, where the appropriate setting is analytic.
The conjectured extremal examples are all quasirandom graphs of a certain density,
hence it is also plausible to expect that there are many 
combinatorially different extremal or close to extremal constructions, and
their combinatorial characterization seems out of reach. 
Yet in analytic language, where graphs are interpreted as symmetric measurable functions on the unit square (called \emph{graphons}), 
they have a simple description as the constant function corresponding to the density.


Aharoni~\cite{aharoni} invoked topological considerations to prove Ryser's Conjecture for $3$-graphs and hence overcame the combinatorial difficulty of having infinitely many extremal structures. 
Our main tasks, the characterization of the extremal $3$-graphs for Ryser's Conjecture (in \cite{HNS2}) and their link-graphs (in the present paper), go a step further in this direction.
The Ryser-extremal hypergraphs are exactly those that satisfy the 
topological condition~(\ref{eq:linkconn}) in
Theorem~\ref{thm:connoflink}. This translation facilitates the 
combinatorial description of the links and eventually of all Ryser-extremal $3$-graphs.

\subsection{The Structure of the Paper}

In the Appendix we collect  background material, reformulated as statements that 
are convenient for our purposes.  
There we give a proof of
Theorem~\ref{thm:deficiencyrainbowmatching}, 
discuss Theorem~\ref{thm:meshulam}, and include
an argument to derive Theorem~\ref{thm:matchconn} from it.
In Section~\ref{sec:ryserforrequals3} we give the proof of Theorem~\ref{thm:connoflink}.

In Section~\ref{thelinkgraph} we prove our main theorem, Theorem~\ref{thm:bipartitegraphchar}. We show that those bipartite graphs whose line graphs are optimal for Theorem~\ref{thm:matchconn} must have a certain form, which we call a CP-decomposition. We show a slightly more general statement involving any \emph{subgraph} of the line graph of a bipartite graph. The precise definition of CP-decomposition in this general setup is given in Section~\ref{thelinkgraph}.

In Section~\ref{sec:goodsets} we prove a theorem that will be crucial for our proof of Theorem~\ref{thm:characterization} in \cite{HNS2}. We define the notion of good sets. Good sets will turn out to be very useful to have in one of the link graphs of a Ryser-extremal $3$-graph. In the main theorem of Section~\ref{sec:goodsets} we show that the lack of good sets in a bipartite graph imposes very strong restrictions on its structure. The proof of this theorem is included in this paper because it uses several of the technical definitions and lemmas introduced for the proof of our main theorem in Section~\ref{thelinkgraph}. 

In the final section we collect several remarks and open problems.

\section{Connectedness of the Link Graph} \label{sec:ryserforrequals3}

In this section we prove Theorem~\ref{thm:connoflink}, which states that the link graph of any Ryser-extremal $3$-graph minimizes  the connectedness of the independence complex of its line graph. On the way we give a reformulated proof of Aharoni's Theorem.

Let $\cH$ be a $3$-partite $3$-graph with vertex classes $V_1$, $V_2$, and $V_3$. We aim to show that $\tau(\cH) \leq 2\nu(\cH)$. To do this, we will consider the link graph (recall Definition~\ref{def:linkg}). We will use the vertex cover number of $\cH$ to find a lower bound on the connectedness of the line graphs of the link graphs, and we will use the matching number of $\cH$ to find an upper bound for at least one link. Combining these bounds will yield the desired inequality $\tau(\cH) \leq 2\nu(\cH)$. 

\begin{prop} \label{prop:linkconn}
Let $\cH$ be a $3$-partite $3$-graph with vertex classes $V_1$, $V_2$, and $V_3$. 
Then for each $i \in \lset{1, 2, 3}$ we have the following:
\begin{enumerate}
	\renewcommand{\theenumi}{(\roman{enumi})}
	\renewcommand{\labelenumi}{\theenumi}
	\item \label{linkconn:lower} For all $S \subseteq V_i$ we have
		\[
			\conn(L(\link{\cH}{S})) \geq \frac{\tau(\cH) - (\abs{V_i} - \abs{S})}{2} - 2.
		\]
	\item \label{linkconn:upper} If $\nu({\cal H}) < |V_i|$, then  
	there is some $S \subseteq V_i$ such that
		\[
			\conn(L(\link{\cH}{S})) \leq \nu(\cH) - (\abs{V_i} - \abs{S}) - 2.
		\]
	\item \label{linkconn:size} If $\nu({\cal H}) < |V_i|$, then for every $S \subseteq V_i$ 
	for which the inequality in \ref{linkconn:upper} holds we have
		\[
			\abs{S} \geq \abs{V_i} - (2\nu(\cH) - \tau(\cH)).
		\]
\end{enumerate}
\end{prop}

\begin{proof}
Let $S \subseteq V_i$. We want to show that the line graph $L(\link{\cH}{S})$ has sufficiently high connectedness. We construct a vertex cover $T_S$ of $\cH$ by taking the vertices in $V_i \setminus S$ and a minimum vertex cover of $\link{\cH}{S}$. This is clearly a vertex cover of $\cH$ because any edge not incident to $S$ intersects $V_i \setminus S$ and any edge incident to $S$ induces an edge in the link of $S$, and hence intersects the vertex cover of the link. We have $\abs{T_S} = \abs{V_i} - \abs{S} + \tau(\link{\cH}{S})$, and since this is a vertex cover, we thus have
\begin{equation} \label{eq:taus}
	\abs{V_i} - \abs{S} + \tau(\link{\cH}{S}) \geq \tau(\cH)
\end{equation}
for all subsets $S \subseteq V_i$. By K\"onig's Theorem, we have $\tau(\link{\cH}{S}) = \nu(\link{\cH}{S})$. We therefore have a lower bound on the matching number of the link graph, and so by Theorem~\ref{thm:matchconn}, we have
\[
	\conn(L(\link{\cH}{S}) \geq \frac{\nu(\link{\cH}{S})}{2} - 2 \geq \frac{\tau(\cH) - (\abs{V_i} - \abs{S})}{2} - 2,
\]
which is the inequality in statement~\ref{linkconn:lower}.

Now we want to show that the inequality in statement~\ref{linkconn:upper} holds for some 
$S$. Suppose to the contrary that for every $S \subseteq V_i$ we had 
$\conn(L(\link{\cH}{S})) \geq \nu(\cH) - (\abs{V_i} - \abs{S}) - 1 
=  \abs{S} - (\abs{V_i} - \nu(\cH) - 1) - 2$. 
Then we can apply Theorem~\ref{thm:deficiencyrainbowmatching} with $d = \abs{V_i} - \nu(\cH) - 1\geq 0$ to get that $\nu({\cal H}) \geq \abs{V_i} - (\abs{V_i} - \nu(\cH) - 1) = \nu(\cH) + 1$, which is a contradiction. Thus some $S \subseteq V_i$ must indeed satisfy the inequality in~\ref{linkconn:upper}.

Now consider such an $S$. Combining the inequalities in~\ref{linkconn:lower} and~\ref{linkconn:upper}, we get
\[
	\frac{\tau(\cH) - (\abs{V_i} - \abs{S})}{2} - 2 \leq \nu(\cH) - (\abs{V_i} - \abs{S}) - 2,
\]
from which the inequality in~\ref{linkconn:size} follows after some rearranging.
\end{proof}

Note that Aharoni's Theorem follows in one line from the above proposition: 
there is an $S \subseteq V_i$ such that $\abs{S} \geq \abs{V_i} - (2\nu(\cH) - \tau(\cH))$, and hence
\[
	\tau(\cH) + \abs{V_i} - \abs{S} \leq 2\nu(\cH).
\]
Since $\abs{V_i} \geq \abs{S}$, we thus have $\tau(\cH) \leq 2\nu(\cH)$ as desired.

We use Proposition~\ref{prop:linkconn} to derive the main theorem of this section.

\begin{proof}[Proof of Theorem~\ref{thm:connoflink}]
Applying Proposition~\ref{prop:linkconn} to $\cH$, we see by~\ref{linkconn:size} that in~\ref{linkconn:upper} equality holds if and only if $S = V_i$ for some $i$. Combining the inequalities in~\ref{linkconn:lower} and~\ref{linkconn:upper} for $S = V_i$ with the fact that $\tau(\cH) = 2\nu(\cH)$ immediately gives that $\conn(L(\link{\cH}{V_i})) = \nu(\cH) - 2$, showing part~\ref{connoflink:conn} of Theorem~\ref{thm:connoflink}. This gives the following chain of inequalities:
\begin{align*}
	\frac{\tau(\cH)}{2} - 2 &= \nu(\cH) - 2 = \conn(L(\link{\cH}{V_i})) \\
	&\geq \frac{\nu(\link{\cH}{V_i})}{2} - 2 = \frac{\tau(\link{\cH}{V_i})}{2} - 2 \\
	&\geq \frac{\tau(\cH)}{2} - 2,
\end{align*}
where the first inequality is valid because of Theorem~\ref{thm:matchconn}, the equality following it is K\"onig's Theorem, and the last inequality is just equation~\eqref{eq:taus} for $S = V_i$. It follows that every inequality is actually an equality, from which part~\ref{connoflink:nu} of Theorem~\ref{thm:connoflink} follows.

From parts~\ref{connoflink:conn},~\ref{connoflink:nu}, and  the fact that $\nu(\cH) = \frac{\tau(\cH)}{2}$, it follows that the link graphs $\link{\cH}{V_i}$ of a Ryser-extremal $3$-graph $\cH$ must be extremal for Theorem~\ref{thm:matchconn}:
\[
	\conn(L(\link{\cH}{V_i})) = \frac{\nu(\link{\cH}{V_i})}{2} - 2.
\]
\end{proof}

\section{The Characterization Theorem} \label{thelinkgraph}

In the main theorem of this section we fully characterize those bipartite graphs for which the connectedness of the line graph is as small as possible, that is, it is equal to two less than half its matching number.

For the proof we need to choose our definitions very subtly and in order to make the induction work, we need to consider a carefully formulated more general statement involving arbitrary \emph{subgraphs} of the line graphs.

\begin{defn}
Let $G$ be a bipartite graph, and let $J \subseteq L(G)$ be a subgraph of its line graph. Two edges of $G$ are called \emph{$J$-adjacent} if they are connected by an edge in $J$, and otherwise \emph{$J$-nonadjacent}. An edge $e \in V(J)$ is \emph{at home} in a subgraph $T \subseteq G$ if $T$ is a path on $4$ vertices, $e$ intersects $T$ in an interior vertex, and $e$ is $J$-adjacent to some edge of $T$.
\end{defn}

\begin{defn}
Let $k \in \N$, let $G$ be a bipartite graph, let $J \subseteq L(G)$ be a subgraph of its line graph, and let $M \subseteq V(J)$ be a matching in $G$ of size $2k$. A \emph{CP-decomposition of $J$ with respect to $M$} is a set of $k$ vertex-disjoint subgraphs $S_1, \dots, S_s, T_1, \dots, T_t$ of $G$ such that
\begin{enumerate}
	\renewcommand{\theenumi}{(\arabic{enumi})}
	\renewcommand{\labelenumi}{\theenumi}
	\item \label{cp:cycles} Each $S_i$ is isomorphic to $C_4$ (a cycle on $4$ vertices), contains two edges of $M$, and every two intersecting edges are $J$-adjacent.
	\item \label{cp:paths} Each $T_j$ is isomorphic to $P_4$ (a path on $4$ vertices), contains two edges of $M$, and every two intersecting edges are $J$-adjacent.
	\item \label{cp:edges} Every edge in $V(J)$ is equal to or parallel to an edge of some $S_i$, or is at home in some $T_j$.
\end{enumerate}
\end{defn}

We call $k = \abs{M}/2$ the \emph{order} of the CP-decomposition. Observe for property~\ref{cp:edges} that the edges of any of the subgraphs $T_j$ are themselves at home in $T_j$ by definition. 

\begin{thm}[CP-Decomposition Theorem] \label{thm:cpdecomposition}
Let $G$ be a bipartite graph, let $J \subseteq L(G)$ be a subgraph of its line graph, and let $M \subseteq V(J)$ be a matching in $G$. If $\conn(J) \leq \frac{\abs{M}}{2} - 2$, then $J$ has a CP-decomposition with respect to $M$.
\end{thm}

Note that by Theorem~\ref{thm:matchconn} we must have that $\conn(J) = \frac{\abs{M}}{2} - 2$, so $\abs{M}$ is even and $V(J)$ does not contain a larger matching than $M$.

First we spell out the special case when $J$ is the entire line graph and prove Theorem~\ref{thm:bipartitegraphchar}. This will provide a characterization of those bipartite graphs $G$ whose line graphs have connectedness as small as possible in terms of the matching number of $G$.

\begin{proof}[Proof of Theorem~\ref{thm:bipartitegraphchar}]
Suppose $\conn(L(G)) = \frac{\nu(G)}{2} - 2$. Then by Theorem~\ref{thm:cpdecomposition}, $L(G)$ has a CP-decomposition, which is a collection of $\nu(G)/2$ pairwise vertex-disjoint subgraphs, each of them a $C_4$ or a $P_4$, such that every edge of $G$ is either an edge of one of the $C_4$'s or is incident to an interior vertex of one of the $P_4$'s.

As it was mentioned in the introduction, the converse of this statement is not used at all in our argument. We include it only to provide a full characterization of the extremal graphs. It will be proven in our follow-up paper~\cite{HNS2}, since the proof uses the concept of home-base hypergraph which is the central concept of that paper.
\end{proof}

The proof of Theorem~\ref{thm:cpdecomposition} is quite involved and will take up the next two subsections. We start with some auxiliary lemmas.

\subsection{Lemmas on $M$-reduced Subgraphs}

For the proof of Theorem~\ref{thm:cpdecomposition} and later we will often use Theorem~\ref{thm:meshulam} in its contrapositive form, which we state here as a corollary.

\begin{cor} \label{cor:roy}
Let $H$ be a graph, let $e \in E(H)$, and let $k \in \N$. If $\conn(H) \leq k$, then either $\conn(H - e) \leq k$ or $\conn(H \explode e) \leq k - 1$.
\end{cor}

In light of Corollary~\ref{cor:roy}, the following definitions will be useful.

\begin{defn}
An edge $e \in E(H)$ is called \emph{decouplable} if $\conn(H - e) \leq \conn(H)$, and \emph{explodable} if $\conn(H \explode e) \leq \conn(H) - 1$.
\end{defn}

By Corollary~\ref{cor:roy} every edge is either decouplable or explodable (or both). In the grand plan of our proof of the CP-decomposition theorem we intend to delete edges of $J\subseteq L(G)$ iteratively until there are no decouplable edges left and hence all edges are explodable (and then we explode one, hence decreasing the connectedness). Crucially, deleting decouplable edges does not increase the connectedness. This explains the following key definition of this subsection.

\begin{defn} \label{def:mreduced}
Let $G$ be a bipartite graph and let $M \subseteq E(G)$ be a matching of it. A subgraph $J \subseteq L(G)$ of the line graph is called \emph{$M$-reduced} if 
\begin{enumerate}
	\renewcommand{\theenumi}{(\arabic{enumi})}
	\renewcommand{\labelenumi}{\theenumi}
	\item $M \subseteq V(J)$, 
	\item $\conn(J) \leq \frac{\abs{M}}{2} - 2$, and 
	\item no edge $ef \in E(J)$ is decouplable.
\end{enumerate}
\end{defn}

Again, note that by Theorem~\ref{thm:matchconn}, if $J$ is $M$-reduced, then $\conn(J) = \frac{\abs{M}}{2} - 2$ and hence $M$ must have an even number of edges.

It will be important to note that if $J$ is $M$-reduced, then $J$ is also $M'$-reduced for any matching $M' \subseteq V(J)$ with $\abs{M'} = \abs{M}$. In particular, if we replace edges of $M$ by parallel edges in $V(J)$, these must share any properties we can deduce for the original edges.

\begin{assm}
For the remainder of the section let $G$ be a bipartite graph, let $M \subseteq E(G)$ be a matching of size $2k$ in $G$, and let $J \subseteq L(G)$ be an $M$-reduced subgraph of the line graph.
\end{assm}

\begin{lem}[Degree Lemma] \label{lem:degree}
For every edge $e \in V(J) \setminus M$ either no edge of $M$ is $J$-adjacent to $e$ or two edges of $M$ are $J$-adjacent to $e$. In particular, if $e$ is parallel to an edge of $M$, then it is not $J$-adjacent to that edge.
\end{lem}

\begin{proof}
Since $J$ is $M$-reduced, we have $\conn(J) \leq k - 2$. Clearly an edge can be $J$-adjacent to at most two edges of $M$ because $M$ is a matching in $G$ and $J \subseteq L(G)$. Suppose for the sake of contradiction that some edge $e \in V(J)$ is $J$-adjacent to $m \in M$, but not $J$-adjacent to any other edge of $M$. Since $me \in E(J)$ and $J$ is $M$-reduced, by Corollary~\ref{cor:roy}, upon exploding $me$ we have $\conn(J') \leq k - 3$ for $J' = J \explode me$. Since $e$ is $J$-adjacent to only one edge from $M$, the explosion keeps $M \setminus \lset{m}$ in $J'$, so $J'$ still contains a matching of size $2k - 1$. Then by Theorem~\ref{thm:matchconn} we have $\conn(J') \geq \frac{2k - 1}{2} - 2 > k - 3$, which is a contradiction. Thus every edge in $V(J)$ is $J$-adjacent to either two edges of $M$ or no edge of $M$.
\end{proof}

\begin{cor} \label{cor:c4adjacencies}
Let $x$, $y$, $x'$, and $y'\in V(G)$ be the vertices of a $C_4$ such that $xy, x'y' \in M$, and $xy', x'y \in V(J)$. Then for every $zy \in V(J)$ with $z \in V(G) \setminus \lset{x, x'}$ we have that $zy$ is $J$-adjacent to $xy$ if and only if it is $J$-adjacent to $x'y$.
\end{cor}

\begin{proof}
Suppose $z \in V(G) \setminus \lset{x, x'}$ with $zy \in V(J)$, and $zy$ is $J$-adjacent to $xy$. Then by the Degree Lemma there is an edge $zw \in M$ which is $J$-adjacent to $zy$. Now consider the matching $M^\times = M \cup \lset{xy', x'y} \setminus \lset{xy, x'y'}$. Note that $\abs{M^\times} = \abs{M}$ and $M^\times \subseteq V(J)$. Applying the Degree Lemma to $M^\times$, we have that since $zw \in M^\times$ is $J$-adjacent to $zy$, also $x'y \in M^\times$ must be  $J$-adjacent to $zy$. The reverse inclusion can be shown by exchanging the roles of $M$ and $M^\times$.
\end{proof}

If $M$ is a matching in a graph, then an \emph{$M$-exposed} vertex is one not in any edge of $M$. A path or cycle is \emph{$M$-alternating} if for every pair of consecutive edges, exactly one of them is in $M$.

\begin{lem}[Alternating Lemma] \label{lem:alternating}
Let $J$ be $M$-reduced, and let $e_1, \dots, e_q \in V(J)$ be the edges of an $M$-alternating path in $G$ starting at an $M$-exposed vertex or the edges of an $M$-alternating cycle in $G$ with $e_q e_1 \notin E(J)$. Then in both cases $e_i e_{i + 1} \notin E(J)$ for all $i = 1, \dots, q - 1$.
\end{lem}

\begin{proof}
\noindent \textbf{Case 1}. $e_1, e_2, \dots, e_q \in V(J)$ are the edges of an $M$-alternating path starting at an $M$-exposed vertex.

Suppose the lemma did not hold and let $j = \min\set{i}{e_i e_{i + 1} \in E(J)}$. If $j$ is odd, then $e_j \notin M$. Since $e_j e_{j + 1} \in E(J)$, by the Degree Lemma there must be another edge of $M$ which is $J$-adjacent to $e_j$. However, $e_1$ has an $M$-exposed vertex, so $j \neq 1$, from which it follows that $e_{j - 1}e_j \in E(J)$, which contradicts the minimality of $j$.

Therefore $j$ is even and $e_j \in M$. Since by assumption $J$ is $M$-reduced, $e_j e_{j + 1}$ is explodable, hence $J' = J \explode e_j e_{j + 1}$ satisfies $\conn(J') \leq k - 3$. Note that since $e_{j - 1}e_j \notin E(J)$, the explosion does not delete $e_{j - 1}$. Thus $M' = M \cup \lset{e_1, e_3, \dots, e_{j - 1}} \setminus \lset{e_2, e_4, \dots, e_j, e_{j + 2}} \subseteq V(J')$ is a matching of size $2k - 1$ (if $j + 2 > q$, let $e_{j + 2}$ be the second edge of $M$ that is $J$-adjacent to $e_{j + 1}$, which exists by the Degree Lemma). This means that by Theorem~\ref{thm:matchconn}, $\conn(J') \geq \frac{2k - 1}{2} - 2 > k - 3$, which is a contradiction. Thus the lemma holds for paths.

\noindent \textbf{Case 2}. $e_1, e_2, \dots, e_q \in V(J)$ are the edges of an $M$-alternating cycle with $e_q e_1 \notin E(J)$.

Since we can reverse the direction of the cycle if necessary, we can assume without loss of generality that $e_q \in M$ and $e_1 \notin M$. If the lemma does not hold, then let $j = \min\set{i}{e_i e_{i + 1} \in E(J)}$. If $j$ is odd, then a reasoning identical to the one in Case 1 yields a contradiction.

Therefore $j$ is even and $e_j \in M$. By assumption, $e_j e_{j + 1}$ is explodable, hence $J' = J \explode e_j e_{j + 1}$ satisfies $\conn(J') \leq k - 3$. We have a matching $M' = M \cup \lset{e_1, e_3, \dots, e_{j - 1}, e_{j + 3}, \dots, e_{q - 1}} \setminus \lset{e_2, e_4, \dots, e_q} \subseteq V(J')$ of size $2k - 1$, so by Theorem~\ref{thm:matchconn}, $\conn(J') \geq \frac{2k - 1}{2} - 2 > k - 3$, which is a contradiction. Thus the lemma also holds for cycles.
\end{proof}

Given two incident non-parallel edges $m \in M$ and $e \in V(J) \setminus M$, we define $\cP_M(m, e)$ to be the set of edges of $M$ which participate in some $M$-alternating path in $G$ starting with $m$, continuing with $e$, and using only edges from $V(J)$. Note that we do not require the edges of the path to be $J$-adjacent. Also note that $m \in \cP_M(m, e)$, and if $me \in E(J)$, then $\cP_M(m, e)$ contains at least one more edge of $M$, namely the other one $J$-adjacent to $e$, which exists by the Degree Lemma.

\begin{lem} \label{lem:longfact}
Let $m \in M$, $e \in V(J) \setminus M$ with $me \in E(J)$, and let $m' \in M$ be the other $M$-edge $J$-adjacent to $e$. Let $W_1$ and $W_2$ be the vertex classes of the bipartite graph $G$, and let $m \cap e \subseteq W_i$. Then for every $m^* \in \cP_M(m, e) \setminus \lset{m, m'}$, there is an edge $g \in V(J)$ for which the following hold:
\begin{enumerate}
	\renewcommand{\theenumi}{(\roman{enumi})}
	\renewcommand{\labelenumi}{\theenumi}
	\item \label{longfact:jadjacent} $g$ is $J$-adjacent to $m^*$,
	\item \label{longfact:sides} $g \cap m^* \subseteq W_{3 - i}$,
	\item \label{longfact:unreachable} If $\hat{m} \in M$ is the other $M$-edge (besides $m^*$) $J$-adjacent to $g$, then $\hat{m} \notin \cP_M(m, e)$.
\end{enumerate}
\end{lem}

\begin{proof}
Suppose not. Then fix $m^* \in \cP_M(m, e) \setminus \lset{m, m'}$ for which the lemma fails. Let $Q = \set{g \in V(J)}{gm^* \in E(J), g \cap m^* \subseteq W_{3 - i}}$. If $Q$ is not empty, then by assumption every edge $g \in Q$ fails property~\ref{longfact:unreachable}.

Since $J$ is $M$-reduced, we have $\conn(J) \leq k - 2$ and when we explode $me$, we get $\conn(J \explode me) \leq k - 3$. We then iteratively delete decouplable edges of $J \explode me$ in an arbitrary order until no edge is decouplable. This results in an $M'$-reduced $J' \subseteq J \explode me$, where $M' = M \setminus \lset{m, m'}$ and $\conn(J') \leq k - 3$ (recall that deleting a decouplable edge does not increase the connectedness). Let $a$ be the vertex in $m' \setminus e$.

Note that if $a$ is not the endpoint of any edge contained in $V(J')$, we are done, since then $\cP_M(m, e) = \lset{m, m'}$, so there is no $m^*$ to choose, and the statement is vacuously true. Thus, assume this is not the case.

We will arrive at a contradiction by showing that $m^*$ is isolated in $J'$, which implies $\conn(J') = \infty$.

First, we show that $m^*$ has no $J'$-neighbors incident to it in $W_{3 - i}$. Take an arbitrary edge $g$ which intersects $m^*$ in $W_{3 - i}$. If $m^*g \notin E(J)$, then $m^*g \notin E(J')$, so we are done. Thus assume $m^*g \in E(J)$, which implies $g \in Q$, and this means that $\hat{m}$, the other $M$-edge $J$-adjacent to $g$ (which exists by the Degree Lemma for $M$ and $J$), is in $\cP_M(m, e)$ by our assumption on $m^*$. If $\hat{m} \in \lset{m, m'}$, then $m^*g \notin E(J')$ by the Degree Lemma applied to $J'$ (because $m, m' \notin V(J')$). Otherwise, there is an $M$-alternating path $e_1, \dots, e_q = \hat{m}$ starting at the vertex $a \in e_1$. This is clearly also an $M'$-alternating path, and since $a$ is an $M'$-exposed vertex in $J'$, the Alternating Lemma (Lemma~\ref{lem:alternating}) applied to $M'$ and $J'$ gives that none of the pairs $e_i$, $e_{i + 1}$ are $J'$-adjacent; in particular, $e_{q - 1}\hat{m} \notin E(J')$. Now there are two cases.

\noindent \textbf{Case 1}. $m^*$ is on this path.

Then the segment of the path starting at $m^*$ and ending with $\hat{m}$, together with $g$, forms an $M'$-alternating cycle. Since $e_{q - 1}\hat{m} \notin E(J')$, the Alternating Lemma tells us that $m^*$ and $g$ are not $J'$-adjacent.

\noindent \textbf{Case 2}. $m^*$ is not on the path.

Then $e_1, \dots, e_q, g, m^*$ is an $M'$-alternating path, and the Alternating Lemma again tells us that $m^*$ and $g$ are not $J'$-adjacent. 

This proves that $m^*$ has no $J'$-neighbor which intersects it in $W_{3 - i}$.

We now show that it also has no $J'$-neighbor intersecting it in $W_i$. Take an arbitrary edge $g$ which intersects $m^*$ in $W_i$. We may again assume $g$ is $J$-adjacent to $m^*$, and hence there is an $\hat{m} \in M$, which is the other $M$-edge $J$-adjacent to $g$. Again, if $\hat{m} \in \lset{m, m'}$, then $m^*g \notin E(J')$ because then $\hat{m} \notin V(J')$ and the Degree Lemma for $J'$ gives that $g$ is not $J'$-adjacent to any edge of $M' = M \cap V(J')$. There is an $M'$-alternating path $e_1, \dots, e_q = m^*$ starting at the vertex $a \in e_1$. Because the path starts at an $M'$-exposed vertex, no two consecutive edges are $J'$-adjacent by the Alternating Lemma. Again there are two cases.

\noindent \textbf{Case 1}. $\hat{m}$ is on this path.

Then the segment of the path starting at $\hat{m}$ and ending with $m^*$, together with $g$, forms an $M'$-alternating cycle. Since $e_{q - 1}m^* \notin E(J')$, the Alternating Lemma tells us that $m^*$ and $g$ are not $J'$-adjacent.

\noindent \textbf{Case 2}. $\hat{m}$ is not on the path.

Then $e_1, \dots, e_q, g$ is an $M'$-alternating path, and the Alternating Lemma will again tell us that $m^* = e_q$ and $g$ are not $J'$-adjacent.

In conclusion, we have shown that $m^*$ does not have any neighbor in $J'$, which was our desired contradiction. Hence no such $m^*$ exists and the proof is complete.
\end{proof}

\begin{lem} \label{lem:c4switch}
Let $m, e, m', f \in V(J)$ be the edges of an $M$-alternating $C_4$ with $m, m' \in M$. Let $M^\times = M \cup \lset{e, f} \setminus \lset{m, m'}$. Then $\cP_{M^\times}(e, m) = \cP_M(m, e) \cup \lset{e, f} \setminus \lset{m, m'}$; in particular, $\abs{\cP_{M^\times}(e, m)} = \abs{\cP_M(m, e)}$.
\end{lem}

\begin{proof}
Let $a$ be the vertex in $m' \cap f$. Any $M$-alternating path starting with $m, e$ must continue with $m'$ and then a path starting at $a$ and never again intersect the vertices of the $C_4$. Similarly, any $M^\times$-alternating path starting with $e, m$ must continue with $f$ and a path starting at $a$ and never again intersect the vertices of the $C_4$. Thus the edges outside of the $C_4$ which are reached will be the same, because the matchings are the same outside the $C_4$.
\end{proof}

\begin{lem} \label{lem:undoswitch}
Let $m, e, m', f \in V(J)$ be the edges of an $M$-alternating $C_4$ with $m, m' \in M$, and let $M^\times = M \cup \lset{e, f} \setminus \lset{m, m'}$. Then $J$ has a CP-decomposition with respect to $M^\times$ if and only if $J$ has a CP-decomposition with respect to $M$.
\end{lem}

\begin{proof}
Suppose $J$ has a CP-decomposition with respect to $M$. We will show that it has a CP-decomposition with respect to $M^\times$. Since the roles of $M$ and $M^\times$ are symmetric, the reverse implication is analogous. Note also that $J$ is $M$-reduced if and only if it is $M^\times$-reduced. There are two cases.

\noindent \textbf{Case 1}. $mem'f$ is a $C_4$ in the CP-decomposition with respect to $M$.

Then $mem'f$ is still an $M^\times$-alternating $4$-cycle and incident edges are $J$-adjacent, so the same CP-decomposition is also a CP-decomposition with respect to $M^\times$.

\noindent \textbf{Case 2}. $mem'f$ is not a $C_4$ in the CP-decomposition with respect to $M$.

Then $m$ and $m'$ must be in either a $C_4$ or a $P_4$ in this decomposition. Suppose first that $m \in S_1$, where $S_1$ is a $C_4$ in the CP-decomposition. Then $m' \notin S_1$, and hence $e, f \notin S_1$. It follows that $e$ and $f$ are neither equal to nor parallel to edges of any $C_4$ in the CP-decomposition, and thus by property~\ref{cp:edges} of the CP-decomposition, they are each at home in some $P_4$ of the CP-decomposition. This means that both endpoints of $m'$ must be interior vertices of some $P_4$. However this is impossible, since $M$-edges are the ending edges of the $P_4$'s of the decomposition, so only one endpoint could be interior. So $m$ is not in a $C_4$ of the decomposition and by symmetry, neither is $m'$.

From now on we assume that $m$ and $m'$ are in two distinct $P_4$'s (they are not in the same $P_4$ because then either $e$ or $f$ would not be at home in any $P_4$). Call these $P_4$'s $T_1$ and $T_2$ with edges $mgh$ and $m'g'h'$, respectively. Note that $m \cap g$ contains an interior vertex, as does $m' \cap g'$, and since $e$ and $f$ must be at home somewhere, one of them, say $e$, is at home in $m \cap g$ and the other ($f$) in $m' \cap g'$ (since $e$ and $f$ are disjoint).

We claim now that replacing $T_1$ and $T_2$ by $egh$ and $fg'h'$ gives us a CP-decomposition with respect to $M^\times$. To check this, we must show that $egh$ and $fg'h'$ are both $P_4$'s whose incident edges are $J$-adjacent, and that every edge which was at home in $T_1$ or $T_2$ is still at home in either $egh$ or $fg'h'$. Both are straightforward consequences of Corollary~\ref{cor:c4adjacencies}, which states that the $J$-neighbors of $e$ and of $m$ at $e \cap m$, which are outside of the $4$-cycle are the same (and likewise for $f$ and $m'$). Thus, $e$ is $J$-adjacent to $g$ and $f$ is $J$-adjacent to $g'$. And any edge which was at home in $T_1$ because it was $J$-adjacent to $m$ is $J$-adjacent also to $e$, and so still at home in $egh$ (and likewise for $T_2$ and $fg'h'$). The only edges left to check are $m$ and $m'$ and edges parallel to $m$, $m'$, $e$, or $f$. Here $m$ and $m'$ are at home in $egh$ and $fg'h'$, respectively, because they are $J$-adjacent to $g$ and $g'$, respectively. Edges parallel to $m$ or $m'$ are $J$-adjacent to $g$ or $g'$, respectively, since they needed to be at home in some $P_4$ of the original CP-decomposition and those are the only possibilities (they are not $J$-adjacent to $m$ or $m'$ by the Degree Lemma, because $J$ is $M$-reduced). Thus they are at home in the new $P_4$'s. All $e$-parallel edges are also $J$-adjacent to $g$, and $f$-parallel ones to $g'$ because of Corollary~\ref{cor:c4adjacencies}. This means that this is indeed a CP-decomposition with respect to $M^\times$, and it is clearly of the same order. This completes the proof.
\end{proof}

\subsection{Proof of the CP-Decomposition Theorem}

We are now ready to start the proof of Theorem~\ref{thm:cpdecomposition}.

\begin{proof}[Proof of Theorem~\ref{thm:cpdecomposition}]
We prove this by induction on $\abs{M}$. Recall that $\abs{M}$ must be even, so write $\abs{M} = 2k$ and proceed by induction on $k$.

For $k = 0$, we have $\conn(J) = -2$, which means $V(J)$ is empty. Thus, it has a CP-decomposition of order $0$, which is an empty collection of cycles and paths.

For $k = 1$, we get $\conn(J) = -1$, so $I(J)$ must have at least two components. Thus there exist two disjoint non-empty subsets $E_1, E_2 \subseteq V(J)$ with $V(J) = E_1 \cup E_2$ such that for all $e_1 \in E_1$ and $e_2 \in E_2$, we have $e_1 e_2 \in E(J)$. By assumption there is a matching $M = \lset{m_1, m_2} \subseteq V(J)$. Since $m_1$ and $m_2$ are not $J$-adjacent (as they are disjoint), they must be in the same component of $I(J)$, and so assume without loss of generality that $m_1, m_2 \in E_1$. Then every edge in $E_2$ is $J$-adjacent to both $m_1$ and $m_2$, and since $G$ is bipartite, every such edge must intersect $m_1$ in one vertex class of $G$ and $m_2$ in the other. Thus the graph formed by the edges in $E_2$ together with $m_1$ and $m_2$ is either a $C_4$ or a $P_4$ together with possibly some parallel edges. If it forms a $C_4$, then the rest of $E_1$ must consist of edges parallel to $m_1$ and $m_2$ because they must be $J$-adjacent to both of the non-$M$ edges of the $C_4$. If the graph is a $P_4$, then the rest of the edges in $E_1$ must be $J$-adjacent to all of the middle edges, and hence are at home in that $P_4$. Therefore, $J$ has a CP-decomposition consisting of a single $C_4$ or $P_4$. This completes the proof for $k = 1$.

Now assume $k \geq 2$. If $\abs{E(J)} = 0$, then $\conn(J) = \infty$, so the statement is vacuously true. So assume $\abs{E(J)} \geq 1$. We may assume $J$ is $M$-reduced so that all edges of $J$ are explodable. (If $J$ is not $M$-reduced, iteratively delete decouplable edges of $J$ until the subgraph is $M$-reduced. A CP-decomposition for the subgraph of $J$ will also be a CP-decomposition of $J$.)

\noindent \textbf{Case 1}. There is an edge $m = ab \in M$ with no $J$-neighbor incident to $a$.

Then there must be a $J$-neighbor $e$ of $m$ incident to $b$, otherwise $m$ is isolated, and $\conn(J) = \infty$, which is a contradiction. Since $J$ is $M$-reduced, when we explode $me \in E(J)$, we have that $J' = J \explode me$ satisfies $\conn(J') \leq k - 3$. By the Degree Lemma for $J$, there is another edge $m' \in M$ which is $J$-adjacent to $e$. Since $M' = M \setminus \lset{m, m'} \subseteq V(J')$ is a matching of size $2k - 2$, we have that $J'$ together with $M'$ satisfy the conditions of the theorem for $k' = k - 1$, so by induction, there is a CP-decomposition of $J'$ with respect to $M'$, say $S_1, \dots, S_s, T_1, \dots T_t$ with $s + t = k - 1$, where each $S_i \cong C_4$ and each $T_j \cong P_4$.

Define $T_{t + 1}$ to be a $P_4$ consisting of the edges $m$, $e$, and $m'$. We claim that $S_1, \dots, S_s, T_1, \dots, T_{t + 1}$ is a CP-decomposition of $J$ with respect to $M$. Since $J' \subseteq J$ and $M' \subseteq M$, the subgraphs $S_i$ form $C_4$'s with two $M$-edges, with intersecting edges $J$-adjacent to each other, and the subgraphs $T_j$ with $j < t + 1$ form $P_4$'s also with this property. The new path $T_{t + 1}$ of course satisfies this as well, so the only thing we still need to check is that the remaining edges are parallel to edges of some $S_i$ or at home in some $T_j$. Clearly, this is already true of the edges in $V(J')$, so consider an edge $f \in V(J) \setminus V(J')$. Then $f \in N_J(m)$ or $f \in N_J(e)$. If $f \in N_J(e)$, then $f$ is at home in $T_{t + 1}$, because both endpoints of $e$ are interior in $T_{t + 1}$. If $f \in N_J(m)$, then $f$ is also at home in $T_{t + 1}$ because $m$ did not have a $J$-neighbor incident to $a$, so $f$ must be adjacent to $m$ at $b$, which is an interior vertex of $T_{t + 1}$. This completes the proof of Case 1.

\noindent \textbf{Case 2}. Every edge in $M$ has a $J$-neighbor on both sides.

Recall that given two incident non-parallel edges $m \in M$ and $e \in V(J) \setminus M$, we define $\cP_M(m, e)$ to be the set of edges of $M$ which participate in some $M$-alternating path in $G$ starting with $m, e$ using edges in $V(J)$. Note that $m \in \cP_M(m, e)$, and if $me \in E(J)$, then $\cP_M(m, e)$ contains at least one more edge of $M$, namely the other one $J$-adjacent to $e$ (which exists by the Degree Lemma).

Let $\cM = \cM(M, J)$ be the smallest family of all matchings $\hat{M} \subseteq V(J)$ with the properties that
\begin{enumerate}
	\renewcommand{\theenumi}{(\arabic{enumi})}
	\renewcommand{\labelenumi}{\theenumi}
	\item $M \in \cM$
	\item For every $\hat{M} \in \cM$ and for every $C_4$ with edges $\hat{m}, \hat{e}, \hat{m}', \hat{f} \in V(J)$, where $\hat{m}, \hat{m}' \in \hat{M}$, we have $\hat{M} \cup \lset{\hat{e}, \hat{f}} \setminus \lset{\hat{m}, \hat{m}'} \in \cM$.
\end{enumerate}
Obviously, each member of $\cM$ can be obtained from $M$ by a finite sequence of the above ``$C_4$-switch'' operation. Observe also that $J$ is $\hat{M}$-reduced for every matching $\hat{M} \in \cM$. 

Let $(M_1, m, e)$ be chosen such that $\abs{\cP_{M_1}(m, e)}$ is maximum among 
\[
	\set{(\hat{M}, \hat{m}, \hat{e})}{\hat{M} \in \cM, \hat{m} \in \hat{M}, \hat{e} \in N_J(\hat{m})}.
\]
Note that the set we are maximizing over is non-empty because we are in Case 2, so $M \in \cM$ has an edge $J$-adjacent to another edge. Our plan is to find a CP-decomposition with respect to $M_1$. This will be enough to prove our theorem because we can then ``undo'' the switches to arrive at our original matching $M$ by repeatedly applying Lemma~\ref{lem:undoswitch}. For convenience we denote the vertex classes of $G$ by $A$ and $B$, with $m \cap e \subseteq A$.

Let $m' \in M_1$ be the other $M_1$-edge $J$-adjacent to $e$. If $m$ has no $J$-neighbor intersecting it in $B$, we may proceed as in Case 1, and thereby have a CP-decomposition with respect to $M_1$. Otherwise, $m$ has a $J$-neighbor on both sides, so let $f$ be a $J$-neighbor of $m$ with $m \cap f \subseteq B$. By the Degree Lemma, $f$ is $J$-adjacent to another edge $m^* \in M_1$. We claim that in fact $m^* = m'$, and hence the edges $m, e, m', f$ form a $C_4$.

Suppose $m^* \neq m'$. If $m^* \notin \cP_{M_1}(m, e)$, we immediately arrive at a contradiction, because $\cP_{M_1}(m^*, f)$ would then properly contain $\cP_{M_1}(m, e)$ (just prepend $m^*, f$ onto any $M_1$-alternating path starting with $m, e$), which contradicts the maximality of $\abs{\cP_{M_1}(m, e)}$. Thus we must have $m^* \in \cP_{M_1}(m, e) \setminus \lset{m, m'}$. By Lemma~\ref{lem:longfact}, there is an edge $g \in V(J) \setminus M_1$ which is $J$-adjacent to $m^*$ with $m^* \cap g \subseteq B$ so that its other $J$-adjacent matching edge, $\hat{m} \in M_1$, is not in $\cP_{M_1}(m, e)$. Then we claim $\cP_{M_1}(\hat{m}, g)$ properly contains $\cP_{M_1}(m, e)$, which would again be a contradiction.

To see that this is the case, take any matching edge $\tilde{m} \in \cP_{M_1}(m, e)$, and we will show that $\tilde{m} \in \cP_{M_1}(\hat{m}, g)$. If an $M_1$-alternating path starting with $m, e$ reaching $\tilde{m}$ contains $m^*$, then we can start with $\hat{m}, g$ and continue along the segment of this path starting at $m^*$, since neither $\hat{m}$ nor $g$ could be used in this path  (otherwise $\hat{m} \in \cP_{M_1}(m, e)$). If, on the other hand, $\tilde{m}$ is reachable from $m, e$ without touching $m^*$, then we may reach $\tilde{m}$ by a path starting with $\hat{m}, g, m^*, f, m, e$. Thus, $\cP_{M_1}(m, e) \subseteq \cP_{M_1}(\hat{m}, g)$, and since the latter contains $\hat{m}$, while the former does not, we have the contradictory proper containment we were hoping for. Therefore $m^* = m'$.

Thus $m$ has only $f$ and edges parallel to $f$ as $J$-neighbors at $B$. We will show now that similarly, $m'$ has only $e$-parallel edges as $J$-neighbors at $B$. By Lemma~\ref{lem:c4switch} applied to $mem'f$, we have $\abs{\cP_{M_1^\times}(e, m)} = \abs{\cP_{M_1}(m, e)}$, so $(M_1^{\times}, e, m)$ is also a maximizing triple, where $M_1^{\times} = M_1 \cup \lset{e, f} \setminus \lset{m, m'}$. Thus, the argument of the previous two paragraphs can be applied to show that $e$ only has $m'$-parallel edges as $J$-neighbors at $B$. By Corollary~\ref{cor:c4adjacencies}, this implies that $m'$ also has only $e$-parallel edges as $J$-neighbors on that side.

We claim that among $m$, $e$, $m'$, $f$, and all parallel edges we have that every parallel pair is non-$J$-adjacent and every pair of intersecting non-parallel edges is $J$-adjacent. To see that two parallel edges are not $J$-adjacent to each other, one must simply apply the Degree Lemma to $M_1$, $M_1^\times$, or one of these with a matching edge switched out for a parallel edge. Now suppose on the contrary that edges $g$ parallel to $m$ and $h$ parallel to $e$ are not $J$-adjacent. Then the Alternating Lemma for $M_1 \cup \lset{g} \setminus \lset{m}$ would imply that $m'$ and $f$ are not $J$-adjacent, which would be a contradiction.

Now we distinguish two further cases.

\noindent \textbf{Case 2(a)}. $mem'f$ and parallel edges form a connected component of $J$.

Then we explode $me$ to yield $J' = J \explode me$ with $\conn(J') \leq k - 3$. Since $J'$ contains the matching $M' = M_1 \setminus \lset{m, m'}$ of size $2k - 2$, $J'$ and $M'$ satisfy the conditions of the theorem with $k' = k - 1$, so by induction, there is a CP-decomposition with respect to $M'$, say $S_1, \dots, S_s, T_1, \dots T_t$ with $s + t = k - 1$.

Define $S_{s + 1}$ to be the $C_4$ given by $mem'f$. It is clear, that adding $S_{s + 1}$ to this CP-decomposition yields a CP-decomposition of $J$. This completes the proof of Case 2(a).

\noindent \textbf{Case 2(b)}. $mem'f$ and parallel edges do not form a component of $J$.

Suppose without loss of generality that there is an edge $g \in V(J)$ not parallel to any of $mem'f$ which is $J$-adjacent to $m$. Note that we must have $m \cap g = m \cap e$ because all the $J$-neighbors of $m$ intersecting it in $m \cap f$ are parallel to $f$. Then we explode $mg$ and iteratively delete all decouplable edges to yield an $M_1'$-reduced $J' \subseteq J \explode mg$ with $\conn(J') \leq k - 3$, where $M_1' = M_1 \setminus \lset{m, m_1}$ with $m_1$ the other $M_1$-edge $J$-adjacent to $g$. Since all $J$-neighbors of $m'$ are parallel to $e$, and they are all $J$-adjacent to $m$, no $J$-neighbors of $m'$ are present in $J'$. So $m'$ has no $J'$-neighbor at $m'\cap e$. Thus it must have some $J'$-neighbor $g'$ at $m' \cap f$, otherwise $m'$ would be isolated and $\conn(J') = \infty$, a contradiction. Thus we explode $m'g'$ and get $J'' = J' \explode m'g'$ with $\conn(J'') \leq k - 4$. Let $m_2 \in M_1$ be the other matching edge $J$-adjacent to $g'$ by the Degree Lemma. Then the matching $M'' = M_1 \setminus \lset{m, m', m_1, m_2}$ of size $2k - 4$ is contained in $J''$. Therefore, $J''$ and $M''$ satisfy the conditions of the theorem for $k'' = k - 2$, so $J''$ has a CP-decomposition with respect to $M''$, say $S_1, \dots, S_s, T_1, \dots, T_t$ with $s + t = k - 2$.

We define $T_{t + 1}$ to be the $P_4$ with edges $\lset{m, g, m_1}$, and $T_{t + 2}$ to be the $P_4$ with edges $\lset{m', g', m_2}$. Then we claim $S_1, \dots, S_s, T_1, \dots, T_{t + 2}$ is a CP-decomposition of $J$ with respect to $M_1$. To see this, we must verify that every edge not in an $S_i$ and not parallel to an edge of an $S_i$ is at home in some $T_j$. This is already true for all edges in $V(J'')$ (since $J'' \subseteq J$), so we only need to consider the edges we have removed by exploding $mg$ and $m'g'$. However, all of these edges were by definition $J$-adjacent (or even $J'$-adjacent) to $m$, $g$, $m'$, or $g'$. The edges $J$-adjacent to $g$ and $g'$ are automatically at home in $T_{t + 1}$ or $T_{t + 2}$ because the vertices of $g$ and $g'$ are the interior vertices of the respective $P_4$'s. However, the only edges $J$-adjacent to $m$ or $m'$ but not at $m \cap g$ or $m' \cap g'$ are parallel to $e$ and $f$. However, $e$-parallel edges are $J$-adjacent to $g$ and $f$-parallel edges are $J$-adjacent to $g'$ by Corollary~\ref{cor:c4adjacencies}, so they are also at home in $T_{t + 1}$ or $T_{t + 2}$. Thus we have a CP-decomposition with respect to $M_1$.

All we need now is to use this CP-decomposition to get a CP-decomposition with respect to our original $M$. This is possible by several applications of Lemma~\ref{lem:undoswitch} because $M_1$ is obtainable from $M$ by a sequence of $C_4$-switches.
\end{proof}

\section{Good Sets} \label{sec:goodsets}

This section introduces the concept of good sets, which (as we will later see in \cite{HNS2}) will help us find the substructure we need in our Ryser-extremal hypergraph in order to prove our characterization theorem by induction. The main result of this section implies that we can find good sets inside our link graphs in several cases, and hence if there are no good sets, we will know that the link graphs must have a certain form.

We start with a graph-theoretic definition, which will form the backbone of the definition of a good set.

\begin{defn} \label{def:decent}
Let $G$ be a bipartite graph with vertex classes $A$ and $B$. A subset $X \subseteq B$ is called \emph{decent} if it satisfies the following conditions:
\begin{enumerate}
	\renewcommand{\theenumi}{(\arabic{enumi})}
	\renewcommand{\labelenumi}{\theenumi}
	\item \label{decent:neighborhood} $\abs{N(X)} \leq \abs{X}$,
	\item \label{decent:nu} $\nu(G) = \abs{N(X)} + \abs{B \setminus X}$,
	\item \label{decent:matchings} For every $x \in X$ and $y\in N(x)$ the edge $xy$ participates in a maximum matching of $G$.
\end{enumerate}
\end{defn}

\begin{lem} \label{lem:decentpaths}
Let $G$ be a bipartite graph with vertex classes $A$ and $B$, and let $M$ be a maximum matching in $G$. Let $X_0 \subseteq B$ be the set of $M$-unsaturated vertices in $B$, and let $X$ be the set of vertices in $B$ reachable on an $M$-alternating path from $X_0$ (including $X_0$). Then $X$ is decent, and $\abs{N(X)} = \abs{X} - \abs{X_0}$.
\end{lem}

\begin{proof}
Let $Y = N(X)$. Then $Y$ is the set of vertices in $A$ reachable on an $M$-alternating path from $X_0$. To see this, consider a vertex $x \in X$ and a neighbor $y \in N(x)$. Either $x$ is unsaturated, in which case $x \in X_0$, so $xy$ is an $M$-alternating path from $X_0$ to $y$, or there is an $M$-alternating path from $X_0$ to $x$, which must end with a matching edge. If $y$ is on this path, we are done. Otherwise, $xy$ is not a matching edge, and hence we can extend our path by the edge $xy$.

We claim that $M$ saturates $Y$ with $(X, Y)$-edges. This is because $M$ is maximum, and thus every $M$-alternating path starting from an unsaturated vertex must end in a saturated vertex, and therefore every vertex of $Y$ is incident to an edge of $M$. Extending the path by such a matching edge must land us in $X$ by definition. Thus this matching edge is an $(X, Y)$-edge. Therefore $\abs{N(X)} = \abs{X} - \abs{X_0} \leq \abs{X}$, so $X$ satisfies property~\ref{decent:neighborhood}. Since $X$ contains all $M$-unsaturated vertices, $M$ saturates $Y$ and $B \setminus X$ with distinct edges, and these are clearly all the edges of $M$. Thus $\nu(G) = \abs{Y} + \abs{B \setminus X}$, so we have~\ref{decent:nu} as well.

We now show that $X$ satisfies~\ref{decent:matchings}. Take an edge $e \in E(G)$ between $X$ and $Y$. If $e \in M$, then we are done. If it has an $M$-unsaturated vertex, then it is only adjacent to one matching edge $m \in M$, and so $M \cup \lset{e} \setminus \lset{m}$ is a maximum matching containing $e$.

Otherwise, $e$ is adjacent to two matching edges $m, m' \in M$. Since $e$ goes between $X$ and $Y$, the vertices of $m$ and $m'$ are reachable by an $M$-alternating path starting from $X_0$. Without loss of generality, the vertex in $m \cap e$ is in $X$. So consider an $M$-alternating path from $X_0$ which ends at that vertex. Note that its last edge is $m$. If $m'$ is not in this path, then we can extend the path by $e$ and $m'$. Switching along this extended path will create a maximum matching containing $e$ (since the path ends at an $M$-unsaturated vertex). If, however, $m'$ was in the original path, then adding $e$ to the path forms an $M$-alternating cycle. Switching the matching along the cycle produces the desired matching. Therefore $X$ is decent, as desired.
\end{proof}

\begin{defn}
Let $G$ be a bipartite graph. A subset $X$ of a vertex class of $G$ is called \emph{equineighbored} if $X$ is nonempty and $\abs{N(X)} = \abs{X}$.
\end{defn}

Note that if $G$ has a perfect matching, then each vertex class is an equineighbored set (unless $G$ is the empty graph).

\begin{lem} \label{lem:alternatingequineighbored}
Let $G$ be a bipartite graph with vertex classes $A$ and $B$ and let $M$ be a perfect matching in $G$. Let $X_0 \subseteq B$, and let $X$ be the set of vertices in $B$ reachable on an $M$-alternating path from $X_0$ (including $X_0$) starting with a non-matching edge. Then $X$ is equineighbored.
\end{lem}

\begin{proof}
Let $Y = N(X)$. Since $M$ is a perfect matching, every $y \in Y$ has a partner $x \in B$ matched to it by $M$. If there is an $M$-alternating path from $X_0$ to $y$ starting with an edge not in $M$, then $x \in X$ because either $x \in X_0 \subseteq X$ or the path can be extended by the matching edge $xy$. If this holds for every $y \in Y$, then there is a matching from $Y$ to $X$, so that $\abs{Y} \leq \abs{X}$, from which $\abs{Y} = \abs{X}$ follows by Hall's Theorem.

Therefore, we need to show that every $y \in Y$ can be reached from $X_0$ by an $M$-alternating path starting with a non-matching edge. Since $y \in N(X)$, it has a neighbor $x \in X$. By the definition of $X$, there is such an $M$-alternating path ending in $x$. If $y$ is on that path, we are done. Otherwise, $xy$ is not an edge of $M$ (because the path to $x$ ends with the matching edge incident to $x$), and so the path could be extended by $xy$, and thus $y$ is on such a path. This concludes the proof.
\end{proof}

\begin{lem} \label{lem:decentequineighbored}
Let $G$ be a bipartite graph with vertex classes $A$ and $B$, and let $M$ be a perfect matching in $G$. Let $X \subseteq B$ be a minimal equineighbored set in $B$. Then $X$ is decent.
\end{lem}

\begin{proof}
$X$ satisfies property~\ref{decent:neighborhood} by being equineighbored. Since $G$ has a perfect matching, there is a matching saturating $B$, and since $\abs{X} = \abs{N(X)}$, we have $\nu(G) = \abs{B} = \abs{N(X)} + \abs{B \setminus X}$, which is~\ref{decent:nu}.

We now show that $X$ satisfies~\ref{decent:matchings}. Let $Y = N(X)$. Let $x \in X$, $y \in Y$, and let $xy \in E(G)$. Fix a perfect matching $M$. Because $N(X) = Y$, it must match $X$ to $Y$. If $xy \in M$, we are done. Otherwise there exist edges $xy', x'y \in M$ adjacent to $xy$. We claim that these edges participate in an $M$-alternating cycle with $xy$, and thus by switching along the cycle we get a new perfect matching which does include $xy$. To show that this happens, consider all $M$-alternating paths starting at $x'$ with a non-matching edge. If there is such a path which hits $y'$, then we can extend the path by $y'x$ and $xy$ to give an $M$-alternating cycle in which $xy$ participates. So assume that no such path hits $y'$. Let $X'$ be the set of $X$-vertices which we can hit on such a path. Then $X'$ is a proper ($x \notin X'$) non-empty ($x' \in X'$) equineighbored subset of $X$ by Lemma~\ref{lem:alternatingequineighbored} applied with $X_0 = \lset{x'}$. This is a contradiction because $X$ was chosen to be minimal.
\end{proof}

\begin{defn} \label{def:good}
Let $G$ be a bipartite graph with vertex classes $A$ and $B$. A subset $X \subseteq B$ is called \emph{good} if it is decent, and if for all $y \in N(X)$ we have $\conn \left( L \left( G - \set{yz \in E(G)}{z \in B \setminus X} \right) \right) > \conn(L(G))$.
\end{defn}

Note in particular that if $X$ is good, then $\set{yz \in E(G)}{z \in B \setminus X} \neq \emptyset$ for all $y \in N(X)$.

\begin{lem} \label{lem:goodsets}
Let $G$ be a bipartite graph with vertex classes $A$ and $B$. Suppose $\nu(G) = 2k$ for some integer $k$ and $\conn(L(G)) = k - 2$. If $G$ has no good set in $A$ nor in $B$, then the following hold:
\begin{enumerate}
	\renewcommand{\theenumi}{(\roman{enumi})}
	\renewcommand{\labelenumi}{\theenumi}
	\item \label{goodsets:perfect} $G$ has a perfect matching
	\item \label{goodsets:minimal} For every minimal equineighbored subset $X \subseteq A$ or $X \subseteq B$ we have $\abs{X} = 2$. In particular, $G[X \cup N(X)]$ is a $C_4$ (possibly with parallel edges).
\end{enumerate}
\end{lem}

Note that the minimality requirement in~\ref{goodsets:minimal} is well-defined because by~\ref{goodsets:perfect} both $A$ and $B$ are equineighbored.

\begin{proof}
Assume that $G$ has no good sets. First, we show that \ref{goodsets:perfect} holds. Suppose $G$ does not have a perfect matching. Let $M$ be a maximum matching in $G$. By assumption, there are some $M$-unsaturated vertices in $A \cup B$. Without loss of generality assume that at least one of them is in $B$. Let $X_0$ be the set of $M$-unsaturated vertices in $B$. Consider all the $M$-alternating paths in $G$ starting from $X_0$. Let $X$ be the set of vertices in $B$ reachable on an $M$-alternating path from $X_0$ (including $X_0$), and let $Y = N(X)$. We claim that $X$ is a good subset. By Lemma~\ref{lem:decentpaths} $X$ is decent, so we must simply check that for all $y \in Y$ we have $\conn \left( L \left( G - \set{yz \in E(G)}{z \in B \setminus X} \right) \right) > \conn(L(G))$.

Let $y \in Y$. Let $G_y = G - \set{yz \in E(G)}{z \in B \setminus X}$. Clearly $M$ is still a maximum matching in $G_y$ and $X_0$ remains the set of $M$-unsaturated vertices. All of the $(X, Y)$-edges have been preserved in $G_y$, so $X$ and $Y$ are still the sets of vertices reachable by an $M$-alternating path from $X_0$. Suppose for the sake of contradiction that we had $\conn(L(G_y)) = k - 2$. Then we pass to an $M$-reduced subgraph $J \subseteq L(G_y)$ of the line graph by iteratively deleting all decouplable edges (see Definition~\ref{def:mreduced}). This means $\conn(J) = k - 2$, but $\conn(J - e) \geq k - 1$ for all $e \in E(J)$).

\begin{claim}
The edges between $X$ and $Y$ form an independent set in $J$.
\end{claim}

\begin{proof}
First, by the Degree Lemma (Lemma~\ref{lem:degree}), any edge $e$ parallel to an edge of $M$ is not $J$-adjacent to any edge of $M$. Next, by the Alternating Lemma (Lemma~\ref{lem:alternating}) any two edges which are together in an $M$-alternating path from $X_0$ are not $J$-adjacent. Now consider a matching edge $m \in M$ and an $(X, Y)$-edge $e$ which intersects it in a vertex $v$. Because $m$ hits $X$, there is an $M$-alternating path starting at $X_0$ which has $m$ as its last edge. If this path ends in $v$, then we can add $e$ to that path to obtain either a longer $M$-alternating path or to obtain an $M$-alternating cycle. Either way, the Alternating Lemma gives that $e$ and $m$ are not $J$-adjacent.

If $v$ is not at the end of this path, then consider the other $M$-edge $m'$ which intersects $e$ (if this does not exist, then $e$ is not $J$-adjacent to $m$ by the Degree Lemma (Lemma~\ref{lem:degree})). There is an $M$-alternating path starting at $X_0$ which has $m'$ as its last edge. This path ends in the intersection of $m'$ and $e$, so by the previous argument, $e$ and $m'$ cannot be $J$-adjacent, and so by the Degree Lemma, $e$ and $m$ are not $J$-adjacent either. Thus we have shown that none of the $(X, Y)$-edges are $J$-adjacent to the edges of $M$.

Now consider two intersecting non-matching edges $e$ and $f$ between $X$ and $Y$. If they were $J$-adjacent, then they would be explodable, but because $e$ and $f$ are not $J$-adjacent to any $M$-edges, $M \subseteq V(J \explode ef)$, so by Lemma~\ref{thm:matchconn}, $\conn(J \explode ef) \geq \abs{M}/2 - 2 = k - 2$. This contradicts explodability, so they must not be $J$-adjacent. 
\end{proof}

Now consider the matching edge $m \in M$ containing $y$. It is isolated in $J$, because all of the edges intersecting $m$ at all are $(X, Y)$-edges. This is a contradiction, because $m$ is then an isolated vertex of $J$, which means $\conn(J) = \infty$, a contradiction. Thus we must have $\conn(L(G_y)) \geq k - 1$ as desired. Thus $X$ is good. This contradicts the assumption that there were no good sets, so $G$ must in fact have a perfect matching.

Now we will show~\ref{goodsets:minimal} holds. Let $X \subseteq B$ be a minimal equineighbored set. We want to show that $\abs{X} = 2$, from which easily follows that the edges incident to $X$ form a $C_4$ (possibly with parallel edges). Indeed, if $X$ is a minimal equineighbored set of size $2$, then its vertices must both have two neighbors (a vertex with only one neighbor would be a proper equineighbored subset, a vertex with more than two neighbors is ruled out by $\abs{N(X)} = 2$, and an isolated vertex is ruled out by the fact that we have a perfect matching), which means they both connect to both neighbors of $X$, which forms a $C_4$.

So suppose that $\abs{X} \neq 2$. We will show that $X$ is good. By Lemma~\ref{lem:decentequineighbored}, $X$ is decent, so we must simply check that for all $y \in N(X)$, the graph $G_y$ formed by erasing from $G$ all edges incident to $y$ and not incident to $X$ has the property that $\conn(L(G_y)) \geq k - 1$.

Indeed suppose it did not. We could then apply Theorem~\ref{thm:cpdecomposition} to get a CP-decomposition of $L(G_y)$. Note that $X$ is still a minimal equineighbored subset of $B$ in $G_y$.

\begin{claim}
$X$ does not contain any interior vertex of a $P_4$ in any CP-decomposition of $L(G_y)$ with respect to any perfect matching.
\end{claim}

\begin{proof}
Fix a perfect matching $M$ of $G_y$, and fix a CP-decomposition $S_1, \dots, S_s$, $T_1, \dots, T_t$ of $L(G_y)$ with respect to $M$. Let $X_0$ be the set of interior vertices of the paths $T_j$ in $X$. Then $X \setminus X_0$ is also equineighbored because the endpoints of the paths $T_j$ which are partnered with the vertices of $X_0$ in the matching $M$ are not in the neighborhood of $X \setminus X_0$ since all edges incident to them must connect to interior vertices of the paths. Since there are $\abs{X_0}$ endpoints in $X$, we have removed at least as many vertices from the neighborhood as we have removed from $X$. Note that $X \setminus X_0$ cannot be empty as $X$ could not have consisted entirely of interior vertices of the paths, since those have at least two distinct neighbors each. It follows that $X_0$ must have been empty and the claim follows.
\end{proof}

\begin{claim}
$X$ does not contain any vertices of a $C_4$ in any CP-decomposition of $L(G_y)$ with respect to any perfect matching.
\end{claim}

\begin{proof}
Fix a perfect matching $M$ of $G_y$, and fix a CP-decomposition $S_1, \dots, S_s$, $T_1, \dots, T_t$ of $L(G_y)$ with respect to $M$. Let $X_0$ be the vertices of some $4$-cycle $S_i$ which are contained in $X$. Then $X \setminus X_0$ is also equineighbored because the two vertices of that $S_i$ which are adjacent to $X_0$ are not in the neighborhood of $X \setminus X_0$ as $X$ does not contain any interior vertices of any $T_j$ by the previous claim, and the only neighbors of the vertices of $S_i$ are other vertices of $S_i$ and interior vertices of paths $T_j$ by the definition of a CP-decomposition. Therefore we would remove at least as many vertices from the neighborhood of $X$ as we would remove from $X$. It follows that if $X_0$ is nonempty, then $\abs{X_0} = 2$, because if $\abs{X_0} = 1$, then we would have $\abs{N(X \setminus X_0)} < \abs{X \setminus X_0}$, which contradicts the fact that $G_y$ has a perfect matching. Since $\abs{X} \neq 2$, we cannot have $X \setminus X_0 = \emptyset$, so $X \setminus X_0$ is a proper equineighbored subset of $X$, which is a contradiction to the minimality of $X$.
\end{proof}

Thus we have shown that $X$ consists entirely of endpoints of $P_4$'s (there are no other types of vertices, since we have a perfect matching). Then $y$ is an interior vertex of some $P_4$. However, $y$ only has neighbors in $X$, so this cannot be the case (since every interior vertex of a path is adjacent to another interior vertex). Since we have reached a contradiction, it follows that we must have $\conn(L(G_y)) \geq k - 1$. Thus $X$ is a good set, which is a contradiction to the conditions of the lemma. Therefore, we must have $\abs{X} = 2$ and $G[X \cup N(X)]$ is a $C_4$, which is \ref{goodsets:minimal}. This proves the lemma.
\end{proof}

\section{Remarks and Open Problems} \label{sec:remarksandopenproblems}

Concerning the tightness of Theorem~\ref{thm:matchconn} several interesting questions remain open. In the main result of our paper we characterized those \emph{bipartite} graphs for which the theorem is tight when $r = 2$.

What happens with this characterization if one leaves out the restriction of bipartiteness? The graph $G$ consisting of a triangle and a hanging edge is an example of a non-bipartite graph which is tight for Theorem~\ref{thm:matchconn}. Indeed, $\nu(G) = 2$ while the line graph is $K_4$ minus an edge, having a disconnected independence complex. It would be very interesting to obtain a full characterizations of those graphs $G$ which are tight for Theorem~\ref{thm:matchconn}. 

Another natural direction is to consider hypergraphs with uniformity higher than $2$. It is not difficult to see that Theorem~\ref{thm:matchconn} is also best possible for every $r > 2$. Just take a matching of size $mr$ and add $m$ edges that intersect $r$ different matching edges each. However, a characterization of those $r$-graphs for which $\conn(\cH) = \frac{\nu(\cH)}{r} - 2$ is still outstanding; the case of $r$-partite $r$-graphs already being very interesting.

A related question concerns the relationship of Theorem~\ref{thm:matchconn} to Ryser's Conjecture for $r > 2$. We mentioned already that in \cite{HNS2} we complete the proof that a graph is tight for Theorem~\ref{thm:matchconn} if and only if it is the link graph of a Ryser-extremal $3$-graph. Is this equivalence or at least one of its directions true for $r > 2$?

Finally, Theorem~\ref{thm:matchconn} has a chance to be best possible only for graphs whose matching number is even. It would be interesting to prove a characterization of $2$-graphs with an \emph{odd} matching number and having a line graph with connectedness as small as possible (in terms of the matching number). 
Is there is a CP-decomposition-type characterization of all (bipartite) graphs with matching number $2k + 1$ and connectedness $k - 1$? 

\section{Appendix} 
\label{sec:appendix}

We start with deriving Theorem~\ref{thm:deficiencyrainbowmatching} from a more general statement from~\cite{aharoniberger} about colored simplicial complexes.

A \emph{coloring} of the vertices of a simplicial complex $\cC$ by colors from a set $X$ is a function $\chi: V(\cC) \to X$. For a subset $S \subseteq X$ of colors, denote by $\cC|_S$ the subcomplex of $\cC$ induced by the vertices which have colors from $S$: that is, let $V(\cC|_S) = \chi^{-1}(S)$ and $\cC|_S = \set{\sigma \in \cC}{\chi(\sigma) \subseteq S}$. 
A simplex is {\em rainbow} if all its vertices have distinct colors.

\begin{thm} \label{thm:deficiencyrainbowsimplex}
Let $\cC$ be a simplicial complex whose vertices are colored with colors from a set $X$, and let $d \geq 0$ be an integer. If for every $S \subseteq X$ we have that $\conn(\cC|_S) \geq \abs{S} - d - 2$, then $\cC$ has a rainbow simplex with $\abs{X} - d$ vertices.
\end{thm}

For the proof of Theorem~\ref{thm:deficiencyrainbowmatching} the crucial thing to note is that if for each hyperedge $xyz \in E(\cH)$ we color the corresponding edge $xy$ of the link graph $\link{\cH}{V_i}$ with the third vertex $z \in V_i$, then a matching in the hypergraph $\cH$ corresponds to a \emph{rainbow matching} (a matching with edges having pairwise distinct colors) in the link graph $\link{\cH}{V_i}$. Then Theorem~\ref{thm:deficiencyrainbowmatching} is an immediate consequence of Theorem~\ref{thm:deficiencyrainbowsimplex} applied with the independence complex $\cI(L(\link{\cH}{V_i}))$ of the link graph. Indeed, $\cI(L(\link{\cH}{V_i}))|_S = \cI(L(\link{\cH}{S}))$ and the vertices of a rainbow simplex in the independence complex of $L(\link{\cH}{V_i})$ correspond to pairwise disjoint edges in the link $\link{\cH}{V_i})$, which extend to pairwise distinct vertices in $V_i$, and hence form a hypergraph matching.

\subsection{Rainbow Simplices}

We now briefly introduce a couple of topological notions which we need for the proof of Theorem~\ref{thm:deficiencyrainbowsimplex}. 

The \emph{join} of two abstract simplicial complexes $\cC$ and $\cD$ is the abstract simplicial complex $\cC * \cD = \set{(\sigma \times \lset{0}) \cup (\tau \times \lset{1})}{\sigma \in \cC, \tau \in \cD}$. A useful fact relating connectedness to joins is the following:

\begin{prop}[Lemma 2.3 in \cite{milnor}] \label{prop:joinconn}
If $\cC$ and $\cD$ are abstract simplicial complexes, then
\[
	\conn(\cC * \cD) \geq \conn(\cC) + \conn(\cD) + 2.
\]
\end{prop}

A map $f: V(\cC) \to V(\cD)$ is a \emph{simplicial map} if the image of each simplex of $\cC$ is a simplex of $\cD$.

If $\cK$ is a simplicial complex, then a \emph{subdivision} of $\cK$ is a simplicial complex $\cK'$ with $\poly{\cK'} = \poly{\cK}$ such that every simplex in $\cK'$ is contained in a simplex in $\cK$.

To determine the connectedness of a simplicial complex, it is sufficient to consider simplicial maps into subdivisions of the simplex.

\begin{prop}[{\cite[Proposition 2.8]{szabotardos}}] \label{prop:subdivfillconn}
A simplicial complex $\cC$ is $k$-connected if and only if for every $j$ with $-1 \leq j \leq k$ and for every simplicial map $f: V(\cS) \to V(\cC)$, where $\cS$ is a subdivision of the boundary of a $(j + 1)$-simplex, there is a subdivision $\cB$ of a $(j + 1)$-simplex with $\cS$ as its boundary, and a simplicial map $\hat{f}: V(\cB) \to V(\cC)$ extending $f$.
\end{prop}

We prove Theorem~\ref{thm:deficiencyrainbowsimplex} using the original proof idea from~\cite{aharonihaxell, aharoniberger} of the case $d=0$, which constructs an appropriate colored triangulation of the simplex and then uses Sperner's Lemma. 
We prove our $d$-defect version with the standard trick of adding a set of $d$ dummy vertices.

\begin{lem}[Sperner's Lemma~\cite{sperner}]
Let $\cT$ be a subdivision of a simplex $\Delta$ of dimension $n$. Let $c: V(\cT) \to A$ be a coloring of the vertices of the subdivision such that
\begin{enumerate}
	\renewcommand{\theenumi}{(\arabic{enumi})}
	\renewcommand{\labelenumi}{\theenumi}
	\item \label{sl:rainbow} Each vertex of $\Delta$ receives a different color,
	\item \label{sl:faces} The vertices of $\cT$ on a face $\sigma$ of $\Delta$ are colored by the colors of the vertices of $\sigma$.
\end{enumerate}
Then there is an $n$-dimensional rainbow simplex in $\cT$.
\end{lem}

\begin{proof}[Proof of Theorem~\ref{thm:deficiencyrainbowsimplex}]
We will prove the statement by induction on $d$. Let first $d = 0$. 

Let $\cC$ be a simplicial complex with a coloring $c: V(\cC) \to X$ of its vertices satisfying the conditions of the theorem and let $\Delta$ be an $(\abs{X} - 1)$-dimensional simplex (so with $\abs{X}$ vertices). The \emph{$k$-skeleton} of $\Delta$ is the subcomplex containing all faces of dimension up to $k$. By induction on $k$, we construct a subdivision $\cT_k$ of the $k$-skeleton of $\Delta$ for every $k = 0, 1, \ldots, \abs{X} - 1$, together with a simplicial map $f_k: V(\cT_k) \to V(\cC)$ so that coloring each vertex $v \in V(\cT_k)$ of the subdivision by the color $c(f_k(v))$ produces a coloring which has property (1) of Sperner's Lemma, as well as property (2) for each face $\sigma$ of $\Delta$ up to dimension $k$. (Such a coloring of will be called a \emph{Sperner coloring}.)

We start with the $0$-skeleton $\cT_0 = \Delta^{(0)}$, which consists of just the vertices of $\Delta$. We choose a simplicial map $f_0: V(\cT_0) \to V(\cC)$ so that every vertex is sent to a vertex with a different color. This is possible because we have as many vertices as there are colors and, most importantly, because the assumption on the connectedness requires that there is a vertex of every color in $\cC$. Indeed, for any $x \in X$, we have $\conn(\cC|_{\lset{x}}) \geq \abs{\lset{x}} - 2 = -1$, hence the subcomplex $\cC|_{\lset{x}}$ is nonempty.

Now suppose that we have already defined a subdivision $\cT_k$ of the $k$-skeleton of $\Delta$ and a simplicial map $f_k: V(\cT_k) \to V(\cC)$ such that if one colors the vertices of the subdivision by the colors of their images under $f_k$, we get a Sperner coloring. We will extend $\cT_k$ and $f_k$ to the $(k + 1)$-skeleton of $\Delta$ by defining the extensions independently for each $(k + 1)$-face $\sigma$ of $\Delta$. The boundary $\partial \sigma$ of $\sigma$ is contained in the $k$-skeleton, so $\cT_k$ contains a subdivision $\cD$ of $\partial \sigma$. Let $S = c(f_k(V(\sigma))) \subseteq X$ be the set of colors of the images of the vertices of $\sigma$ under $f_k$. Because $f_k$ induces a Sperner coloring, we must have that $\abs{S} = k + 2$ and $f_k(V(\cD)) \subseteq \cC|_S$. By assumption, $\conn(\cC|_S) \geq \abs{S} - 2 = k$, and since $\cD$ is a subdivision of the boundary of a $(k + 1)$-simplex, by Proposition~\ref{prop:subdivfillconn} there is a subdivision $\cE$ of $\sigma$ with $\cD$ as its boundary, and a simplicial map $f_\sigma: V(\cE) \to V(\cC|_S)$ extending $f_k$. Doing  this for each $(k + 1)$-simplex one after another, we obtain a subdivision $\cT_{k + 1}$ of the $(k + 1)$-skeleton and a map $f_{k + 1}: V(\cT_{k +1}) \to V(\cC)$ defined as the union of all the maps $f_\sigma$ with $\sigma$ ranging over the $(k + 1)$-faces of $\Delta$. Since each $f_\sigma$ agrees with $f_k$ on the boundary, the union agrees with $f_k$ on the $k$-skeleton and it is well-defined. Also, $f_{k + 1}$ induces a Sperner coloring by construction.

Continuing in this manner, we end up with a subdivision $\cT_{\abs{X} - 1} = \cT$ of the entire simplex $\Delta$ and a simplicial map $f: V(\cT) \to V(\cC)$ inducing a Sperner coloring. Hence, by Sperner's Lemma, there is a rainbow simplex $\tau$ in $\cT$ with $\abs{X}$ vertices. The colors of $V(\tau)$ were defined as the colors of its image via $f$, hence the simplex of $\cC$ with vertices $f(V(\tau))$ must also have $\abs{X}$ vertices with all different colors. So we found our rainbow simplex, which concludes the proof for $d = 0$.

Let now $d \geq 1$ and let $\cC$ be a simplicial complex with a coloring $c: V(\cC) \rightarrow X$ of its vertices such that for every $S \subseteq X$ we have that $\conn(\cC|_S) \geq \abs{S} - d - 2$. Our strategy is to add some new vertices and new simplices to $\cC$ to get a complex $\hat{\cC}$ and extend the coloring $c$ to $\hat{\cC}$ such that $\hat{\cC}$ satisfies the conditions of the theorem with $d_{\hat{\cC}} = d - 1$. We will then apply the induction hypothesis to find a rainbow simplex in $\hat{\cC}$, and since it will turn out that it may contain at most one new vertex, removing it will yield a rainbow simplex in $\cC$ with at least $\abs{X} - d$ vertices.

For each $x \in X$, let $v^{(x)}$ be a new vertex which we color by $x$. Let $\cN$ be the simplicial complex consisting of the isolated vertices $\set{v^{(x)}}{x \in X}$, and let $\hat{\cC} = \cC * \cN$. We claim that $\hat{\cC}$ fulfills the conditions of the theorem with $d_{\hat{\cC}} = d - 1$. Indeed, applying Proposition~\ref{prop:joinconn} we get that $\conn(\hat{\cC}|_S) \geq (\abs{S} - d - 2) - 1 + 2 = \abs{S} - (d - 1) - 2$ for every $S \subseteq X$. Here we used that $\hat{\cC}|_S = \cC|_S * \cN|_S$ and that $\conn(\cN|_S) = -1$, as each color is represented among the new vertices, so $\cN|_S$ is non-empty. Thus, by induction, $\hat{\cC}$ contains a rainbow simplex $\tau$ with $\abs{X} - d + 1$ vertices. To complete the proof of the theorem we just need to recall that no two vertices of $\cN$ form a simplex, hence $\tau$ can contain at most one of the new vertices. Thus there is a face of $\tau$ spanned by at least $\abs{X} - d$ vertices from $\cC$, providing the rainbow simplex we were looking for.
\end{proof}

\subsection{The Independence Complex}

Meshulam~\cite{meshulam} proved a homological version of Theorem~\ref{thm:meshulam}, where everywhere in the statement $\conn$ is replaced by the homological connectedness $\conn_H$. He used  the Mayer-Vietoris sequence and the observation that, provided $G$ is simple, $\cI(G - e) = \cI(G) \cup (e * \cI(G \explode e))$ and $\cI(G) \cap (e * \cI(G \explode e))$ is the suspension of $\cI(G \explode e)$. (Once proved for simple graphs, Theorem~\ref{thm:meshulam} follows easily for arbitrary $G$.) Adamaszek and Barmak~\cite{adamaszekbarmak}, mostly concerned with a question of Aharoni, Berger, and Ziv~\cite{aharonibergerziv}, proved that the $\conn$ on the right hand side of inequality~\eqref{eq:meshulamtheorem} can be replaced with the following function $\psi$:
\[
	\psi(G) = \left\{ \begin{array}{ll}
	-2 & G = \emptyset \\
	+\infty & V(G) \neq \emptyset,  E(G) = \emptyset \\
	\max_{e \in E(G)} \min \lset{\psi(G - e), \psi(G \explode e) + 1}& \mbox{otherwise.}
	\end{array}
	\right.
\]
It can be easily seen by induction on $\abs{E(G)}$ that Theorem~\ref{thm:meshulam} implies the theorem of Adamaszek and Barmak~\cite{adamaszekbarmak}, but there seems to be no direct way to derive the implication in the other direction. However, the \emph{proof} in~\cite{adamaszekbarmak} can easily be modified to give Theorem~\ref{thm:meshulam}. One simply takes $e$ to be an arbitrary edge, defines $k = \min(\conn(G - e), \conn(G \explode e) + 1)$, and proceeds as in~\cite{adamaszekbarmak} to show that the homological connectedness of $G$ is at least $k$. To conclude that $\conn(G) \geq k$, one only needs to show that $k \geq 1$ implies that $\cI(G)$ is simply connected and then appeal to the Hurewicz Theorem. This can be done in an argument identical to the one in \cite{adamaszekbarmak}.

One can apply Theorem~\ref{thm:meshulam} to prove Theorem~\ref{thm:matchconn}.

\begin{proof}[Proof of Theorem~\ref{thm:matchconn}]
We proceed by induction on $\abs{E(J)}$. If $J$ contains an isolated vertex, the lemma is trivially true, since then $\conn(J) = \infty$. Thus we may assume that every vertex of $J$ has a neighbor. If $M = \emptyset$, the lemma is trivially true, since the connectedness of anything is at least $-2$, so assume $\abs{M} \geq 1$. Now consider an edge $m \in M \subseteq V(J)$. This edge (vertex of $J$) has a neighbor $e$ in $J$. Since $M \subseteq V(J - me)$, by induction we have $\conn(J - me) \geq \abs{M}/r - 2$. Now consider what happens when we explode $me$. We remove from $V(J)$ all neighbors of $m$ and $e$. Since $m \in M$, none of the neighbors of $m$ are in $M$, and since $e$ has size at most $r$, it intersects at most $r$ edges of $M$ (one of them being $m$). Therefore, $V(J \explode me)$ still contains a matching of size at least $\abs{M} - r$. By induction, we then have $\conn(J \explode me) \geq (\abs{M} - r)/r - 2 = \abs{M}/r - 3$. Applying Theorem~\ref{thm:meshulam}, we obtain
\[
	\conn(J) \geq \min \lset{\conn(J - me), \conn(J \explode me) + 1} \geq \frac{\abs{M}}{r} - 2,
\]
which is what was wanted.
\end{proof}


\begin{thebibliography}{10}

\bibitem{abuprok}
A.~Abu-Kazneh and A.~Pokrovskiy, \emph{Intersecting extremal constructions in Ryser's Conjecture for $r$-partite hypergraphs}, submitted

\bibitem{adamaszekbarmak}
M.~Adamaszek and J.~A.~Barmak, \emph{On a lower bound for the connectivity of the independence complex of a graph}, Discrete Mathematics \textbf{311} (2011), 2566-2569.

\bibitem{aharoni}
R.~Aharoni, \emph{Ryser's conjecture for tri-partite $3$-graphs}, Combinatorica \textbf{21} (2001), no. 1, 1-4.

\bibitem{aharonibaratwanless}
R.~Aharoni, J.~Bar\'at and I.~Wanless, \emph{Multipartite hypergraphs achieving equality in Ryser's Conjecture}, submitted

\bibitem{aharoniberger}
R.~Aharoni and E.~Berger, \emph{The intersection of a matroid and a simplicial complex}, Trans. Amer. Math. Soc. \textbf{358} (2006), 4895-4917. 

\bibitem{aharonibergerziv}
R.~Aharoni, E.~Berger, and R.~Ziv, \emph{Independent systems of representatives in weighted graphs}, Combinatorica \textbf{27} (2007), no. 3, 253-267.

\bibitem{aharonichudnovskykotlov}
R.~Aharoni, M.~Chudnovsky, and A.~Kotlov, \emph{Triangulated spheres and colored cliques}, Discrete and Computational Geometry \textbf{28} (2) (2002), 223-229.

\bibitem{aharonihaxell}
R.~Aharoni and P.~Haxell, \emph{Hall's theorem for hypergraphs}, J. Graph Theory \textbf{35} (2000), no. 2, 83-88.

\bibitem{babai-frankl}
L. Babai and P. Frankl, \emph{Linear Algebra Methods in Combinatorics}, (September 1992), University of Chicago.

\bibitem{bergerziv}
E.~Berger and R.~Ziv, \emph{A note on the cover number and independence number in hypergraphs}, manuscript.

\bibitem{berlekamp}
E.~R.~Berlekamp, \emph{On subsets with intersections of even cardinality}, Canad. Math. Bull. \textbf{12} (1969), 471-474.

\bibitem{fhmw-2015} N.~Franceti\'c, S.~Herke, B.~D.~McKay, I.~M.~Wanless, 
\emph{On Ryser's Conjecture for Linear Intersecting Multipartite Hypergraphs}, arxiv.org/abs/1508.00951.

\bibitem{fueredi}
Z.~F\"uredi, \emph{Maximum degree and fractional matchings in uniform hypergraphs}, Combinatorica \textbf{1} (1981), 155-162.

\bibitem{haxell}
P.~Haxell, \emph{Independent transversals and hypergraph matchings - an elementary approach}, in Recent Trends in Combinatorics (A.~Beveridge, J.~Griggs, L.~Hogben, G.~Musiker, P.~Tetali, eds), IMA Volume in Mathematics and its Applications, Springer 2016,  215-233. 

\bibitem{HNS2}
P. Haxell, L.~Narins, T. Szab\'o, \emph{Extremal Hypergraphs for Ryser's Conjecture II: Home-base Hypergraphs}, submitted, http://arxiv.org/abs/1401.0171.

\bibitem{haxellscott}
P.~Haxell and A.~D.~Scott, \emph{On Ryser's conjecture}, The Electronic Journal of Combinatorics 19.1 (2012), paper 23.


\bibitem{lovasz}
L.~Lov\'asz, \emph{On minimax theorems of combinatorics}, Matematikai Lapok \textbf{26} (1975), 209-264 (in Hungarian).

\bibitem{mansoursongyuster}
T.~Mansour, C.~Song, and R.~Yuster, \emph{A comment on Ryser's conjecture for intersecting hypergraphs}, Graphs and Combinatorics \textbf{25} (2009), 101-109.

\bibitem{meshulam}
R.~Meshulam, \emph{Domination numbers and homology}, J. Combin. Theory Ser. A 102 (2003), 321-330.

\bibitem{milnor}
J.~Milnor, \emph{Construction of Universal Bundles, II}, Ann. of Math \textbf{63}, 430-436.

\bibitem{lotharsthesis}
L.~Narins, Ph.D Thesis, FU Berlin, \emph{Extremal Hypergraphs for Ryser's Conjecture}, Freie Universit\"at Berlin, 2014.

\bibitem{ryser}
H.~Ryser, \emph{Neue probleme der kombinatorik}, Vortr\"age \"uber Kombinatorik Oberwolfach, Mathematisches Forschungsinstitut Oberwolfach, Colloquia Mathematica Societatis J\'anos Bolyai, 1967, 69-91.

\bibitem{sidorenko} 
A.F. Sidorenko, \emph{A correlation inequality for bipartite graphs}, Graphs and Combinatorics {\bf 9} (1993), 
201-204.

\bibitem{simonovits}
M. Simonovits, \emph{Extremal graph problems, degenerate extremal problems and super-saturated graphs}, Progress in Graph Theory (Waterloo, Ont., 1982), Academic Press, Toronto, ON (1984), 419-437.

\bibitem{sperner}
E.~Sperner, \emph{Neuer Beweis f\"ur die Invarianz der Dimensionszahl und des Gebietes}, Abh. Math. Sem. Univ. Hamburg \textbf{6} (1928), 265-272.

\bibitem{szabotardos}
T.~Szab\'o and G.~Tardos, \emph{Extremal problems for transversals in graphs with bounded degree}, Combinatorica \textbf{26} (2006), 333-351.

\bibitem{tuza1}
Zs.~Tuza, \emph{Some special cases of Ryser's conjecture}, manuscript, 1979.

\bibitem{tuza2}
Zs.~Tuza, \emph{Ryser's conjecture on transversals of $r$-partite hypergraphs}, Ars Combinatoria \textbf{16} (1983), 201-209.

\end{thebibliography}
\end{document}